\documentclass[11pt]{article}
\usepackage{etex}
\usepackage{amsmath,amssymb,amsfonts,amsthm}
\usepackage{a4wide}
\usepackage{color}
\usepackage{verbatim}
\usepackage{epsfig,subfigure,epstopdf}
\usepackage[]{graphics}
\usepackage{graphicx,inputenc}
\usepackage{subfigure}
\usepackage[justification=centering]{caption}
\usepackage{pictex}
\usepackage{wrapfig}
\usepackage{multirow}
\usepackage{amscd,amssymb,stmaryrd}
\usepackage{amsmath}
\usepackage{graphicx}
\usepackage{amstext}
\usepackage{pstricks,pst-node,pst-plot,pst-coil}
\usepackage{amsthm}
\usepackage{amsmath}
\usepackage{color}
\usepackage{mathtools}
\usepackage{hyperref}
\usepackage[english]{babel}
\usepackage{booktabs}
\usepackage{mathrsfs}
\usepackage{comment}
\usepackage{float}
\usepackage[linesnumbered,ruled,vlined, ruled,vlined]{algorithm2e}

\theoremstyle{exampstyle}
\newtheorem{theorem}{Theorem}
\newtheorem{lemma}{Lemma}[section]
\newtheorem{corollary}[lemma]{Corollary}

\newtheorem{assumption}[lemma]{Assumption}

\numberwithin{equation}{section}

\newbox\boxaddrone \newbox\boxaddrtwo

\def\N+{n\in\mathbb{N}^{+}}

\def\n{\partial{\overrightarrow{\bf n}}}

\def\L{\mathcal{L}}

\def\l{\langle}
\def\ro{\rangle_{\Omega}}
\def\rp{\rangle_{\kappa,\partial \Omega}}

\def\re{\operatorname{Re}}
\def\s{\text{Span}}

\def\1d{\mathcal{D}((-\Delta)^{\gamma+1})}
\def\A{\mathcal{A}}

\title{Theoretical and numerical studies of inverse source problem for the linear parabolic equation with sparse boundary measurements}
\author{ Guang Lin\footnote{Department of Mathematics and Mechanical Engineering, Purdue University, West Lafayette, IN 47906, USA. (Email: guanglin@purdue.edu)}, ~Zecheng Zhang \footnote{Department of Mathematics, Purdue University, West Lafayette, IN 47906, USA. (Email: zecheng.zhang.math@gmail.com)}, ~Zhidong Zhang \footnote{School of Mathematics (Zhuhai), Sun Yat-sen University, Zhuhai 519082, Guangdong, China. (Email: zhangzhidong@mail.sysu.edu.cn)} }

\begin{document}
\maketitle

\begin{abstract}
\noindent We consider the inverse source problem in the parabolic equation, where the unknown source possesses the semi-discrete formulation. Theoretically, we prove that the flux data from any nonempty open subset of the boundary can uniquely determine the semi-discrete source. This means the observed area can be extremely small, and that is why we call the data as sparse boundary data. For the numerical reconstruction, we formulate the problem from the Bayesian  sequential  prediction perspective and conduct the numerical examples which estimate the space-time-dependent source state by state. To better demonstrate the performance of the method, we solve two common multiscale problems from two models with a long sequence of the source. The numerical results illustrate that the inversion is accurate and efficient.
\end{abstract}



\section{Introduction.}
\subsection{Background and literature. }
As a classical type of PDEs, the parabolic equation is widely applied in physics, engineering, finance, and so on. The inverse source problems in the parabolic equation have various applications in the real world, and the corresponding research has a long history. We list several representative works here \cite{Cannon:1968,HettlichRundell:2001,Ikehata:2007,Isakov:1990,JinKianZhou:2021,Yamamoto:2009}. Denoting the unknown source by $F(x,t)$, to recover it, we need the observations of the solution $u$ on the whole domain $\mathbb R^d\times (0,\infty)$, which is impractical in almost every situation.  Therefore, in the research of the inverse problem, people usually consider some special cases of the unknown source. For instance, a popular case is the variable separable source, i.e., $F(x,t):=p(x)q(t)$. Given the spatial component $p(x)$ or the temporal component $q(t)$, recovering the other unknown part is a classical field in the inverse problems. See \cite{ChengLiu:2020,HettlichRundell:2001,HuangImanuvilovYamamoto:2020,RundellZhang:2018} and the references therein.  Furthermore, the work \cite{RundellZhang:2020} considered the case when $p(x)$ and $q(t)$ are both unknown; in \cite{JinKianZhou:2021} the authors recovered the unknown source $p(x',t)$ where $x'\in \mathbb R^{d-1}$. Comparing with the variable separable form, the semi-discrete formulation below simulates the general source $F(x,t)$ better:   
\begin{equation*}
 F(x,t):=\sum_{k=1}^K p_k(x)\chi_{{}_{t\in[t_{k-1},t_k)}}.
\end{equation*}
We can see the above formulation can approximate the general form $F(x,t)$ accurately if the time mesh $\{0\le t_0<t_1<\cdots\}$ is fine enough. The parabolic equation with a space-time-dependent source has many applications. For instance, in the area of medical research, one needs to trace the blood distribution in some tissues of the human body; in the reservoir simulation area, one example is to trace the amount of liquid injected into oil field consisting of the impermeable rocks; in the ocean, people may need to determine the location of a leaking oil tanker and so on. In \cite{LiZhang:2020}, the inverse source problem with the semi-discrete source is investigated. 

In this work, the unknown source still has semi-discrete formulation. The measurements we used are the boundary flux data, which means the observed area will be the subset of the boundary of the domain. To save cost, absolutely we want the observed area to be as small as possible. In \cite{LiZhang:2020,RundellZhang:2020}, the authors considered the heat equation on the two-dimensional unit disc, and proved the uniqueness theorem under the boundary flux data from two chosen points on the boundary. The proof depended on the explicit representation of the eigensystem of the Laplacian $\Delta$ on the two-dimensional unit disc. The conclusions in  \cite{LiZhang:2020,RundellZhang:2020} confirm that in the heat equation, if the domain has a smooth shape, then it is possible to recover the source from sparse boundary data.  Sequentially, there is a natural question that can we solve  the similar inverse source problem in the parabolic equation with a general domain, in which the explicit representation of the eigensystem is not applicable.

\subsection{Mathematical statement and main theorem.} 
We give the mathematical model as follows:
 \begin{equation}\label{PDE}
 \begin{cases}
  \begin{aligned}
    (\partial_t+\A)u(x,t)&=\sum_{k=1}^K p_k(x)\chi_{{}_{t\in[t_{k-1},t_k)}}, &&(x,t)\in\Omega\times(0,\infty),\\
    u(x,t)&=0,&&(x,t)\in\partial\Omega\times(0,\infty)\cup \Omega\times\{0\}, 
  \end{aligned}
  \end{cases}
 \end{equation}
 where $\Omega\subset\mathbb{R}^d$ is open and bounded. The operator $\A$ is a symmetric elliptic operator, defined as  
 \begin{equation}\label{A}
 \A \psi=-\nabla\cdot(\kappa(x)\nabla \psi)+q(x)\psi,\quad \psi\in H^2(\Omega)\cap H_0^1(\Omega).
 \end{equation}
 Here $\kappa$ and $q$ possess appropriate regularities to support the future proof and they satisfy 
 $$0<C_1\le \kappa(x) \le C_2<\infty\ \text{and}\ q(x)\ge 0\ \text{for a.e.}\ x\in\Omega.$$
 In the source term, the spatial components $\{p_k(x)\}_{k=1}^K$ and the time mesh $\{t_k\}_{k=0}^K$ are all undetermined. The measurements we used are the boundary flux data:   
 $$\frac{\partial u}{\n}(x,t),\ (x,t)\in\Gamma\times(0,\infty)\subset\partial\Omega\times(0,\infty),$$
 where $\overrightarrow{\bf n}$ means the outward normal unit vector of the boundary and $\Gamma\subset\partial\Omega$ is the observed area. Hence, the interested inverse problem in this work is given as follows:
 $$\text{recovering}\ \{t_0,t_k,p_k(x)\}_{k=1}^K\ \text{in the source from the data}\ \frac{\partial u}{\n}\big|_{\Gamma\times(0,\infty)}.$$ 
 
 For this inverse problem, we have two goals: first, proving that the boundary flux data generated from a small observed area $\Gamma$ can uniquely determine the source; second, recovering the unknown source from the sparse boundary data numerically.  
In the aspect of theoretical analysis, we attempt to build the uniqueness theorem. 
Firstly, we give some prior assumptions for the semi-discrete unknown source $\sum_{k=1}^K p_k(x)\chi_{{}_{t\in[t_{k-1},t_k)}}$.  
\begin{assumption}\label{condition_f}
\hfill
\begin{itemize}
\item [(a)] For the time mesh $\{t_k\}_{k=0}^K$, we have $K\in\mathbb{N}^+\cup \{\infty\},$ and there exists $\eta>0$ such that \\$ \inf\{|t_k-t_{k+1}|:k=0,\cdots,K-1\}\ge \eta.$  
\item [(b)] $\{p_k(x)\}_{k=1}^K\subset L^2(\Omega)$, 
$\|p_k\|_{L^2(\Omega)}\ne 0$ for $k=1,\cdots, K$, and $\|p_k-p_{k+1}\|_{L^2(\Omega)}\ne 0$ for $k=1,\cdots, K-1$.
\end{itemize}
\end{assumption}
In Assumption \ref{condition_f} (a), we do not require the time mesh $\{t_k\}_{k=0}^K$ be finite, which means $K$ can be infinity. Also, Assumption \ref{condition_f} (b) is set to make sure the source  
 $\sum\nolimits_{k=1}^K p_k(x)\chi_{{}_{t\in[t_{k-1},t_k)}}$ can not 
 be simplified further. Otherwise, if $\|p_{k_0}\|_{L^2(\Omega)}  =\|p_{k_1-1}-p_{k_1}\|_{L^2(\Omega)}=0$, then we can write the source as 
$$\sum\nolimits_{k\notin \{k_0, k_1-1,k_1\}}p_k(x)\chi_{{}_{t\in[t_{k-1},t_k)}}+p_{k_1}(x) \chi_{{}_{t\in[t_{k_1-2},t_{k_1})}}.$$

In the next section, we will give Assumption \ref{condition_f_regularity}. With Assumptions \ref{condition_f} and \ref{condition_f_regularity}, we can state the main theorem of this work.  
\begin{theorem}\label{uniqueness}
Under Assumptions \ref{condition_f} and \ref{condition_f_regularity}, the flux data from any nonempty open subset of $\partial \Omega$ can uniquely determine the semi-discrete unknown source $\sum\nolimits_{k=1}^K p_k(x)\chi_{{}_{t\in[t_{k-1},t_k)}}$. 

More precisely, given two sets of unknowns $\{t_0,t_k,p_k(x)\}_{k=1}^K$ and $\{\tilde t_0,\tilde t_k,\tilde p_k(x)\}_{k=1}^{\tilde K}$, we denote the corresponding solutions by $u$ and $\tilde u$ respectively, assume Assumptions \ref{condition_f} and \ref{condition_f_regularity} be valid, and let $\Gamma\subset\partial\Omega$ be an arbitrary nonempty open subset. Provided $$\frac{\partial u}{\n}\big|_{\Gamma\times(0,\infty)}=\frac{\partial \tilde u}{\n}\big|_{\Gamma\times(0,\infty)},$$ 
we have $\{t_0,t_k,p_k(x)\}_{k=1}^K=\{\tilde t_0,\tilde t_k,\tilde p_k(x)\}_{k=1}^{\tilde K}$.
\end{theorem}
Theorem \ref{uniqueness} confirms that the data from any nonempty open subset of the boundary is sufficient to support the uniqueness of this inverse source problem.    

\subsection{Bayesian formulation and outline.} 
For this inverse source problem, after proving Theorem \ref{uniqueness}, it is time to consider reconstructing the unknown source numerically. The design of the algorithm will be challenging since there are too many unknowns in the source. 
The conventional methods are hard to handle problems with high dimensional unknowns and are very sensitive to the observation locations.
However, from the semi-discrete formulation,  one needs to estimate a sequence of unknown states $\{p_k(x)\}_{k = 1}^K$ recursively in time.
One choice is to estimate each state $p_k(x)$
given a sequence of observations up to $k$ in time by the posterior distribution. This task of sequential prediction based on the online observations can then be categorized as the standard Bayesian filtering problems which have been thoroughly studied in the control theory \cite{efendiev2006preconditioning, chung2020multi, stuart2010inverse, lin2021multi}.
This motivates us to formulate the inversion problem in the Bayesian framework.

The numerical reconstruction is a time series tracking problem and consists of constructing the posterior probability distribution of a sequence of states based on observations.
The resulting model typically is highly nonlinear with non-Gaussian distribution, and hence obtaining an explicit exact solution is a hard work. This reminds us of estimating the state recursively via a filtering approach \cite{jouin2016particle}. 
In this work, we are going to use the particle filter (PF) algorithm to compute the posterior of the state in a sequence.
The benefit of the particle filter algorithm is: it can handle any nonlinear properties and is capable to solve any distribution (no Gaussian assumption) \cite{van2000unscented, doucet2009tutorial}. The PF algorithms usually involve sampling from a probabilistic model which depends on the previous states and observations.
The samples are called particles and constitute
an empirical approximation of the model. 
The trajectories are then extended by sampling the next state particles based on the existing particles and new observations. Degeneration issue is a concern of the PF algorithm; but the resampling techniques can reduce the variance of the particles. In this work, we use the systematical resampling \cite{douc2005comparison, doucet2009tutorial} and please check \cite{douc2005comparison} for a comprehensive review of the resampling schemes. Another concern of the PF algorithms is the design of the proposal.
Many techniques like the linearization \cite{doucet2000sequential}, extended Kalman filter, and unscented PF \cite{van2000unscented} are among the most important options to this issue.
In this work, we use the transition prior, and the results are satisfactory. The details can be seen in Section \ref{numerical_sec}.

The contribution of this work can be summarized as follow:
\begin{enumerate}
\item Prove the uniqueness theorem rigorously, and confirm that the semi-discrete source can be recovered by the flux data from any nonempty open subset of the boundary. This conclusion can decrease the cost of the real projects and is meaningful in practical applications.  

\item To solve the inverse problem numerically, we formulate the process as a sequential prediction and hence use the filter algorithm to predict the semi-discrete source. The method enables us to solve the models with a long sequence of states and is very accurate.
\end{enumerate}

Finally, we give the outline of this paper. In Section \ref{section_pre}, we collect several useful results and show some lemmas which will be used in future proofs. In Section \ref{section_unique} we will prove the main result--Theorem \ref{uniqueness}. The proof depends on the Laplace transform and the knowledge of complex analysis. In Section \ref{numerical_sec}, we consider the numerical reconstruction of this inverse source problem. The algorithm is given,  and some satisfactory numerical results are provided. In Section \ref{section_concluding}, we list some future works for this inverse source problem.

\section{Preliminaries.}\label{section_pre}
\subsection{Eigensystem of $\A$ on $\Omega$.}\label{eigen}
For the operator $\A$ defined in \eqref{A}, we denote the eigensystem by $\{\lambda_n,\varphi_n(x)\}_{n=1}^\infty$. Since $\A$ is positive definite and self-adjoint, we have that $0<\lambda_1\le \lambda_2\le\cdots$ and $\lambda_n\to\infty$ as $n\to \infty$, and the set of eigenfunctions $\{\varphi_n(x)\}_{n=1}^\infty$  constitutes an orthonormal basis of the space $L^2(\Omega)$. 
Furthermore, if $\varphi_n$ is an eigenfunction of $\A$ corresponding to $\lambda_n$, so is $\overline{\varphi_n}$. Here $\overline{\varphi_n}$ means the complex conjugate of $\varphi_n$. This gives that the set $\{\varphi_n(x)\}_{n=1}^\infty$ coincides with $\{\overline{\varphi_n(x)}\}_{n=1}^\infty$. 
Also, from the trace theorem, we have $\{\frac{\partial\varphi_n}{\n}|_{\partial \Omega}\}_{n=1}^\infty\subset H^{1/2}(\partial \Omega)$, and the next lemma concerns the non-vanishing property of $\frac{\partial\varphi_n}{\n}|_{\partial \Omega}$.

\begin{lemma}\label{nonempty_open}
If $\Gamma$ is a nonempty open subset of $\partial \Omega$, then for each $n\in\mathbb N^+$, $\frac{\partial \varphi_n}{\n}$ can not vanish almost everywhere on $\Gamma$. 
\end{lemma}
\begin{proof}
Assume that there exists $k\in \mathbb{N}^+$ and a nonempty open subset $\Gamma\subset\partial\Omega$ such that $\frac{\partial \varphi_k}{\n}=0$ a.e. on       $\Gamma$. This implies that we can find $r>0$ and $x_0\in\partial\Omega$ such that $\frac{\partial\varphi_k}{\n}\equiv 0$ a.e. on $B(x_0,r)\cap\partial \Omega$, where $B(x_0,r)$ is  the ball centered at $x_0$ with radius $r$ in $\mathbb{R}^d.$ 
Now define the extension $\Omega_e$ of $\Omega$ as $\Omega_e=\Omega\cup B(x_0,r)$ and the extended function $\varphi_{k,e}$ of $\varphi_k$ on $\Omega_e$ as
\begin{equation*}
 \varphi_{k,e}(x)=
 \begin{cases}
 \begin{aligned}
  &\varphi_k (x), && x\in \Omega,\\
  &0,&&x\in B(x_0,r)\setminus \Omega.
  \end{aligned}
 \end{cases}
\end{equation*}
Obviously, $\varphi_{k,e}\in H^1(\Omega_e)$. Now for any $\psi\in C_c^\infty(\Omega_e),$ we have 
\begin{equation*}
 \begin{aligned}
  \int_{\Omega_e} (\kappa\nabla\varphi_{k,e}\cdot\nabla\psi-\lambda_k\varphi_{k,e}\psi)\ dx
  &=\int_\Omega (\kappa\nabla\varphi_{k,e}\cdot\nabla\psi-\lambda_k\varphi_{k,e}\psi)\ dx +\int_{B(x_0,r)\setminus \Omega} 0\ dx\\
  &=\int_\Omega [(\A-\lambda_k)\varphi_k]\psi\ dx=0.
 \end{aligned}
\end{equation*}
Hence, $\varphi_{k,e}\in H^1(\Omega_e)$ is a weak solution of the equation 
\begin{equation*}
 (\A-\lambda_k) v(x)=0,\ x\in \Omega_e,
\end{equation*}
and we can find $B(x_0',r')\subset B(x_0,r)\subset \Omega_e$ s.t. 
$\varphi_{k,e}\equiv0$ on $B(x_0',r')$. By the unique continuation principle, we have $\varphi_{k,e}\equiv 0$ on $\Omega_e,$ i.e. $\varphi_k$ vanishes on $\Omega,$ which contradicts with the fact that $\|\varphi_k\|_{L^2(\Omega)}=1.$ Therefore we conclude that $\frac{\partial\varphi_k}{\n}$ does not vanish almost everywhere on any nonempty open subset of $\partial \Omega$, and the proof is complete.
\end{proof}

\subsection{The set $\{\xi_l\}_{l=1}^\infty$.}
In this part, we build the auxiliary set of functions $\{\xi_l\}_{l=1}^\infty$ which will be used later. 
For $\psi_1,\psi_2\in L^2(\Omega),\ \psi_3,\psi_4\in L^2(\partial \Omega),$ we define 
\begin{equation*}
\l\psi_1,\psi_2\ro=\int_\Omega\psi_1(\cdot)\overline{\psi_2(\cdot)}\ dx, 
\ \ \l\psi_3,\psi_4\rp=\int_{\partial \Omega}\kappa(\cdot)\psi_3(\cdot)\overline{\psi_4(\cdot)}\ dx,
\end{equation*}
and accordingly we give the induced norm $\|\cdot\|_{L^2(\Omega)},\|\cdot\|_{L^2(\kappa,\partial \Omega)}$ as
$$\|\cdot\|^2_{L^2(\Omega)}=\l\cdot,\cdot\ro,\ \|\cdot\|^2_{L^2(\kappa,\partial \Omega)}=\l\cdot,\cdot\rp.$$ 
Not hard to see $\|\cdot\|_{L^2(\kappa,\partial \Omega)}$ is equivalent to the usual $L^2$ norm $\|\cdot\|_{L^2(\partial \Omega)}$. The next lemma considers a density property.
\begin{lemma} 
For the eigenfunctions $\{\varphi_n\}_{n=1}^\infty$ defined in Section \ref{eigen}, we have  $\text{Span}\{\frac{\partial\varphi_n}{\n}|_{\partial \Omega}\}_{n=1}^\infty$ is dense in $L^2(\kappa,\partial \Omega)$. 
\end{lemma}
\begin{proof}
 With the density of $H^{3/2}(\partial\Omega)$ in $L^2(\partial \Omega)$ under the norm $\|\cdot\|_{L^2(\partial\Omega)}$, it is sufficient to prove that $\tilde\psi\in H^{3/2}(\partial \Omega)$ vanishes almost everywhere on $\partial \Omega$ if $\l \tilde\psi,\frac{\partial\varphi_n}{\n}|_{\partial \Omega}\rp=0$ for each $n\in\mathbb{N}^+$.
 
We let $\psi$ satisfy the following system: 
\begin{equation*}
\begin{cases}
\begin{aligned}
    \A \psi(x)&=0, &&x\in \Omega,\\
    \psi&=\tilde \psi,&&x\in\partial \Omega.
\end{aligned}
\end{cases}
\end{equation*}
The regularity $\tilde\psi\in H^{3/2}(\partial \Omega)$ gives that $\psi\in H^2(\Omega)$, then the Green's identity can be used. Fixing  one $n\in\mathbb N^+$, we have 
\begin{equation*}
\l \A \psi,\varphi_n\ro-\l \psi, \A\varphi_n\ro 
=\l \tilde\psi,\frac{\partial\varphi_n}{\n}\rp-\l\frac{\partial\psi}{\n}, \varphi_n\rp.
\end{equation*} 
With the results $\A \psi=0$ on $\Omega$, $\varphi_n=0$ on $\partial\Omega$ and $\l\tilde\psi,\frac{\partial\varphi_n}{\n}\rp=0$, the above equality gives that $$\l \psi, \A\varphi_n\ro =\lambda_n\l \psi,\varphi_n\ro=0.$$ 
So we see that for each $n\in\mathbb N^+$, $\l\psi, \varphi_n\ro=0$. This together with the completeness of $\{\varphi_n\}_{n=1}^\infty$ in $L^2(\Omega)$ yields that $\|\psi\|_{L^2(\Omega)}=0$. 
From the definition of weak derivative, it is not hard to see that the first-order weak derivative of $\psi$ is also vanishing on $\Omega$. This gives that $\|\psi\|_{H^1(\Omega)}=0$. By the continuity of the trace operator, it holds that 
$$\|\psi|_{\partial \Omega}\|_{L^2(\partial \Omega)}=\|\tilde\psi\|_{L^2(\partial \Omega)}\le C\|\psi\|_{H^1(\Omega)}=0.$$
So we have $\tilde\psi=0$ almost everywhere on $\partial\Omega$ and the proof is complete. 
\end{proof}

Now we can build the auxiliary set $\{\xi_l\}_{l=1}^\infty$.
With the density of $\text{Span}\{\frac{\partial\varphi_n}{\n}|_{\partial \Omega}\}_{n=1}^\infty$ in $L^2(\kappa,\partial \Omega)$, we can construct the orthonormal basis $\{\tilde \xi_l\}_{l=1}^\infty$ in $L^2(\kappa,\partial \Omega)$ as follows. Firstly we set $\tilde\xi_1=\frac{\partial\varphi_1}{\n}|_{\partial \Omega}/\|\frac{\partial\varphi_1}{\n}\|_{L^2(\kappa,\partial \Omega)},$ and for $l=2,3,\cdots$, we assume that the orthonormal set $\{\tilde\xi_j\}_{j=1}^{l-1}$ has been constructed. Then we let $n_l\in\mathbb N^+$ be the smallest number such that $\frac{\partial\varphi_{n_l}}{\n}|_{\partial \Omega}\notin \s\{\tilde\xi_j\}_{j=1}^{l-1}$, and after that we can pick $\tilde\xi_l\in \s\{\frac{\partial\varphi_{n_l}}{\n}|_{\partial \Omega},\ \tilde\xi_1, \cdots, \tilde\xi_{l-1}\}$ such that 
\begin{equation*}
 \l\tilde\xi_l, \tilde\xi_j\rp=0\ \text{for}\ j=1,\cdots, l-1,\ \ \|\tilde\xi_l\|_{L^2(\kappa,\partial \Omega)}=1.
\end{equation*}
The density of $\text{Span}\{\frac{\partial\varphi_n}{\n}|_{\partial \Omega}\}_{n=1}^\infty$ in $L^2(\kappa,\partial \Omega)$ ensures that $\{\tilde\xi_l\}_{l=1}^\infty$ is an orthonormal basis in $L^2(\kappa,\partial \Omega)$. Also, from the smoothness property $\frac{\partial\varphi_n}{\n}|_{\partial  \Omega}\in H^{1/2}(\partial \Omega)$ we have 
$\{\tilde\xi_l\}_{l=1}^\infty\subset  H^{1/2}(\partial \Omega)$.

Next, for each $l\in \mathbb N^+$, let $\xi_l\in H^1( \Omega)$ be the weak solution of the following system: 
\begin{equation}\label{PDE_xi}
\begin{cases}
\begin{aligned}
  \A\xi_l(x)&=0, &&x\in \Omega,\\
    \xi_l&=\tilde \xi_l,&&x\in\partial \Omega.
\end{aligned}
\end{cases}
\end{equation}
Then we get the auxiliary set $\{\xi_l\}_{l=1}^\infty$ on $ \Omega$.

\subsection{The coefficients $\{c_{z,n}\}$ and Assumption \ref{condition_f_regularity}.} 
The proof of Theorem \ref{uniqueness} relies on Assumption \ref{condition_f_regularity}. Before to state Assumption \ref{condition_f_regularity}, we need to build the coefficients $\{c_{z,n}\}$.  
For a fixed point $z\in\partial \Omega$, we define the series 
$\psi_z^N\in H^1(\overline\Omega)$ as 
\begin{equation}\label{psi_z^N}
 \psi_z^N(x)=\sum_{l=1}^N \tilde\xi_l(z)\overline{\xi_l(x)}, \ x\in  \overline\Omega,
\end{equation}
where $\{\tilde\xi_l,\xi_l\}_{l=1}^\infty$ are defined in the above subsection. 
Then we consider the following system and denote the solutions by $\{u_z^N\}$:   
\begin{equation}\label{u_z^N}
\begin{cases}
 \begin{aligned}  
  (\partial_t+\A)u_z^N(x,t)&=0, &&(x,t)\in \Omega\times(0,\infty),\\
 u_z^N(x,t)&=0, &&(x,t)\in \partial \Omega\times(0,\infty),\\
  u_z^N (x,0)&=-\psi_z^N, &&x\in \Omega.
 \end{aligned}
\end{cases}
\end{equation}
The finite summation $\psi_z^N$ is constructed following the role of Dirac delta function, which reflects the information of the targeted function on a specific point via integration. 
With the following lemma, we can give the coefficients  $\{c_{z,n}\}$. 
\begin{lemma}\label{c_z,n}
 For each $z\in\partial \Omega$ and $n\in \mathbb N^+$, $\lim_{N\to \infty}\l \psi_z^N,\overline{\varphi_n}\ro$ exists. 
\end{lemma}
\begin{proof}
 Firstly from Green's identities we have 
 \begin{equation*}
 \begin{aligned}
  \l \psi_z^N,\overline{\varphi_n}\ro&=\lambda_n^{-1}\l\psi_z^N, \A\overline{\varphi_n}\ro \\
  &=\lambda_n^{-1}\Big( \l \A\psi_z^N,\overline{\varphi_n}\ro-\l\psi_z^N,\frac{\partial\overline{\varphi_n}}{\n}\rp+\l\frac{\partial\psi_z^N}{\n},\overline{\varphi_n}\rp\Big) \\
  &=-\lambda_n^{-1}\sum_{l=1}^N \tilde\xi_l(z)\l\frac{\partial\varphi_n}{\n}, \tilde\xi_l\rp=:c_{z,n}^N,
  \end{aligned}
 \end{equation*}
where the system \eqref{PDE_xi} and the boundary condition of $\varphi_n$ are used. From the definition of $\{\tilde\xi_l\}_{l=1}^\infty$, we see that                $\l\frac{\partial\varphi_n}{\n}, \tilde\xi_l\rp=0$ for large $l$. Hence if $N$ is sufficiently large, the value $\lambda_n^{-1}\sum_{l=1}^N \tilde\xi_l(z)\l\frac{\partial\varphi_n}{\n}, \tilde\xi_l\rp$ will not depend on the value of $N$. This means that $\lim_{N\to \infty}\l \psi_z^N,\overline{\varphi_n}\ro$ exists and the proof is complete. 
\end{proof}

Now we define 
\begin{equation}\label{c_z,n}
c_{z,n}:=\lim_{N\to \infty}\l\psi_z^N,\overline{\varphi_n}\ro
=\lim_{N\to \infty}c_{z,n}^N,\quad  
p_{k,n}:=\l p_k(\cdot),\varphi_n(\cdot)\ro,
\end{equation}
it is time to state Assumption \ref{condition_f_regularity}.   
\begin{assumption}\label{condition_f_regularity}
\hfill
\begin{itemize}
\item [(a)] For $k\in\{1,\cdots,K\}$ and a.e. $z\in\partial\Omega$, there exists $C>0$ which is independent of $N$ such that $\sum_{n=1}^\infty  |c^N_{z,n}p_{k,n}|<C<\infty$ for each $N\in\mathbb N^+$. 
\item [(b)] For a.e. $z\in \partial \Omega$, it holds that $\sum_{k=1}^K\sum_{n=1}^\infty |c_{z,n}p_{k,n}|<\infty$.
\item [(c)] For a.e. $t\in[0,\infty)$, the spatial components $\{p_k(x)\}_{k=1}^K$ are smooth enough such that the series  
$\sum_{l=1}^\infty\tilde\xi_l(x)\l\frac{\partial u}{\n}(\cdot, t),\tilde\xi_l(\cdot)\rp$ converges to $\frac{\partial u}{\n}(x, t)$ pointwisely for a.e. $x\in\partial\Omega$.  
\end{itemize}
\end{assumption}

\section{Uniqueness theorem.}
\label{section_unique}
\subsection{Representation for the boundary flux data.}
In this part, we will consider how to present the boundary flux $\frac{\partial u}{\n}|_{\partial \Omega}$ by the unknown source. 
Firstly, we give the next lemma.  
\begin{lemma}\label{data_1}
 For a.e. $z\in\partial\Omega$, we define $w_z^N=u_z^N+\psi_z^N,$ where $u_z^N$ and $\psi_z^N$ are given by \eqref{u_z^N} and \eqref{psi_z^N}, respectively. Then we have: 
 \begin{align*}
  &\text{for a.e.}\ t\in[0,t_0], \\
  &\qquad\qquad-\frac{\partial u}{\n}(z,t)=0;\\
  &\text{for a.e.}\ t\in (t_{m_0-1},t_{m_0}]\ \text{with}\ 1\le m_0\le K, \\
  &\qquad\qquad-\frac{\partial u}{\n}(z,t)=\lim_{N\to\infty}\int_\Omega p_{m_0}(x) w_z^N(x,t-t_{m_0-1})\ dx\\
  &\qquad\qquad\qquad\qquad\qquad+\sum_{m=1}^{m_0-1}\lim_{N\to\infty}\int_\Omega p_m(x) [w_z^N(x,t-t_{m-1})- w_z^N(x,t-t_m)]\ dx;\\
  &\text{for a.e.}\ t\in(t_K,\infty)\ \text{(when K is finite)}, \\
  &\qquad\qquad -\frac{\partial u}{\n}(z,t)=\sum_{m=1}^K\lim_{N\to\infty}\int_\Omega p_m(x) [w_z^N(x,t-t_{m-1})- w_z^N(x,t-t_m)]\ dx.
 \end{align*}
\end{lemma}
\begin{proof}
 From the definition of $\psi_z^N$, we see that 
 $ (\partial_t+\A)\psi_z^N=0.$ This result and \eqref{u_z^N} show that $w_z^N$ satisfies the equation 
\begin{equation*}
 (\partial_t+\A)w_z^N=0, \  (x,t)\in \Omega\times(0,\infty),
\end{equation*}
with the boundary condition $w_z^N|_{\partial \Omega}=\psi_z^N|_{\partial \Omega}$ and the initial condition $w_z^N(x,0)=0$. Therefore, Green's identities give that for each $v\in H_0^1(\Omega)$, 
\begin{equation*}
 \int_\Omega (\partial_t+q)w_z^N(x,t) v(x)
 +\kappa(x)\nabla w_z^N(x,t) \cdot \nabla v(x)\ dx=0,\ \ t\in (0,\infty).
\end{equation*}

From \eqref{PDE}, we obtain 
\begin{equation*}
 \begin{aligned}
  \int_0^t \int_\Omega F(x,\tau) w_z^N(x,t-\tau)\ dx\ d\tau=\int_0^t\int_\Omega (\partial_t+\A)u(x,\tau)\ w_z^N(x,t-\tau)\ dx\ d\tau. 
 \end{aligned}
\end{equation*}
Green's identities and the vanishing initial conditions of $u$ and $w_z^N$ give that 
\begin{equation*}
 \begin{aligned}
 \int_0^t\int_\Omega (\partial_t+q) u(x,\tau)\ w_z^N(x,t-\tau)\ dx\ d\tau
 =& \int_0^t\int_\Omega (\partial_t+q) w_z^N(x,t-\tau)\ u(x,\tau)\ dx\ d\tau,\\
 \int_0^t\int_\Omega -\nabla\cdot(\kappa\nabla u(x,\tau))\ w_z^N(x,t-\tau)\ dx\ d\tau
 =& \int_0^t\int_\Omega \kappa(x)\nabla u(x,\tau)\cdot\nabla w_z^N(x,t-\tau)\ dx\ d\tau \\
 &-\int_0^t \int_{\partial \Omega} \kappa(x)\frac{\partial u}{\n}(x,\tau)\psi_z^N(x)\ dx\ d\tau.
\end{aligned}
\end{equation*}
Hence we have  
 \begin{align*}
  &\int_0^t \int_\Omega F(x,\tau) w_z^N(x,t-\tau)\ dx\ d\tau\\
  &=\int_0^t\int_\Omega (\partial_t+q) w_z^N(x,t-\tau)\ u(x,\tau)
  +\kappa(x)\nabla w_z^N(x,t-\tau)\cdot\nabla u(x,\tau) \ dx\ d\tau\\
 &\quad -\int_0^t \int_{\partial \Omega} \kappa(x)\frac{\partial u}{\n}(x,\tau)
 \psi_z^N(x)\ dx\ d\tau\\
 &=-\int_0^t \sum_{l=1}^N \tilde\xi_l(z)\ \big\l\frac{\partial u}{\n}(\cdot,\tau),\tilde\xi_l(\cdot) \big\rp \ d\tau.
 \end{align*}

 For $t\in (t_{m_0-1},t_{m_0}]$ with $1\le m_0\le K$, the left side of the above equality can be written as 
 \begin{align*}
  \int_0^t \int_\Omega F(x,t-\tau) w_z^N(x,\tau)\ dx\ d\tau
  =&\int_0^{t-t_{m_0-1}} \int_\Omega p_{m_0}(x) w_z^N(x,\tau)\ dx\ d\tau\\
  &+\sum_{m=1}^{m_0-1}\int_{t-t_m}^{t-t_{m-1}} \int_\Omega p_m(x) w_z^N(x,\tau)\ dx\ d\tau.
 \end{align*}
 Similarly, if $t>t_K$, we have 
  \begin{align*}
  \int_0^t \int_\Omega F(x,t-\tau) w_z^N(x,\tau)\ dx\ d\tau
  =\sum_{m=1}^K\int_{t-t_m}^{t-t_{m-1}} \int_\Omega p_m(x) w_z^N(x,\tau)\ dx\ d\tau.
 \end{align*}
 Then for the above equality, differentiating at time $t$ gives that for $t\in (t_{m_0-1},t_{m_0}]$ with $1\le m_0\le K$, 
 \begin{align*}
  -\sum_{l=1}^N \tilde\xi_l(z)\ \big\l\frac{\partial u}{\n}(\cdot,t),\tilde\xi_l(\cdot) \big\rp =&\int_\Omega p_{m_0}(x) w_z^N(x,t-t_{m_0-1})\ dx\\
  &+\sum_{m=1}^{m_0-1}\int_\Omega p_m(x) [w_z^N(x,t-t_{m-1})- w_z^N(x,t-t_m)]\ dx,
 \end{align*}
and for $t>t_K$ (From Assumption \ref{condition_f}, we see that $K$ is finite), 
 \begin{equation*}
  -\sum_{l=1}^N \tilde\xi_l(z)\ \big\l\frac{\partial u}{\n}(\cdot,t),\tilde\xi_l(\cdot) \big\rp=\sum_{m=1}^K\int_\Omega p_m(x) [w_z^N(x,t-t_{m-1})- w_z^N(x,t-t_m)]\ dx. 
 \end{equation*}

From Assumption \ref{condition_f_regularity} $(c)$ we have for a.e. $t>t_K$, 
 \begin{equation*}
  -\frac{\partial u}{\n}(z,t)=\sum_{m=1}^K\lim_{N\to\infty}\int_\Omega p_m(x) [w_z^N(x,t-t_{m-1})- w_z^N(x,t-t_m)]\ dx, 
 \end{equation*}
 and for a.e. $t\in (t_{m_0-1},t_{m_0}]$ with $1\le m_0\le K$, 
 \begin{align*}
  -\frac{\partial u}{\n}(z,t)=&\lim_{N\to\infty}\int_\Omega p_{m_0}(x) w_z^N(x,t-t_{m_0-1})\ dx\\
  &+\sum_{m=1}^{m_0-1}\lim_{N\to\infty}\int_\Omega p_m(x) [w_z^N(x,t-t_{m-1})- w_z^N(x,t-t_m)]\ dx.
 \end{align*}
 
 Assumption \ref{condition_f} $(a)$ ensures that for a fixed $t$, the summations in the right sides of the above results are finite. Therefore, the order of summation and limit can be exchanged. Moreover, the proof for $t\in[0,t_0]$ is trivial. The proof is complete. 
\end{proof}

From the above lemma, the following corollary can be deduced.
\begin{corollary}\label{data_formula}
For a.e. $z\in\partial \Omega$ and $t\in(0,\infty)$, it holds that 
\begin{equation*}
 -\int_0^t \frac{\partial u}{\n}(z, \tau)\ d\tau
 =\int_0^t \sum_{k=1}^K \chi_{{}_{t-\tau\in [t_{k-1},t_k)}} 
 \Big[\sum_{n=1}^\infty c_{z,n}p_{k,n} 
 (1-e^{-\lambda_n\tau})\Big]\ d\tau.
 \end{equation*}
 Here $p_{k,n}$ and $c_{z,n}$ are defined in \eqref{c_z,n}.  
\end{corollary}
\begin{proof} 
Firstly, let us evaluate the limit $\lim_{N\to \infty}\l  p_k(\cdot), \overline{w_z^N(\cdot,t)}\ro$ for $k=1,\cdots, K$ and $t\in(0,\infty)$.  
 Fixing $N\in \mathbb{N}^+$, from the definition of $\psi_z^N$ we have 
 $\psi_z^N\in L^2(\Omega)$. So the Fourier expansion of $\psi_z^N$ can be given as $\psi_z^N=\sum_{n=1}^\infty c_{z,n}^N \overline{\varphi_n}$, where $c_{z,n}^N$ is defined in the proof of Lemma \ref{c_z,n}. Moreover, from $w_z^N=u_z^N+\psi_z^N$ and equation \eqref{u_z^N}, we have  
 \begin{equation*}
  w_z^N(x,t)=\sum_{n=1}^\infty c_{z,n}^N (1-e^{-\lambda_nt})\overline{\varphi_n(x)}.
 \end{equation*}
 The above result together with the regularity $\psi_z^N\in L^2(\Omega)$ yields that $w_z^N(\cdot,t)\in L^2(\Omega)$ for $t\in [0,\infty)$. Also recall that $p_k\in L^2(\Omega)$, then   
 \begin{equation*}
 \l p_k(\cdot), \overline{w_z^N(\cdot,t)}\ro
  = \sum_{n=1}^\infty c_{z,n}^N p_{k,n} (1-e^{-\lambda_nt}).
 \end{equation*}
 The condition in Assumption \ref{condition_f_regularity} implies that the Dominated Convergence Theorem can be applied to the above series. So we have  
\begin{equation*}
 \lim_{N\to\infty}\sum_{n=1}^\infty c_{z,n}^N p_{k,n} (1-e^{-\lambda_nt})= \sum_{n=1}^\infty \lim_{N\to\infty}c_{z,n}^N p_{k,n} (1-e^{-\lambda_nt}).
\end{equation*}
 From \eqref{c_z,n}, we have  
 \begin{equation*}
 \lim_{N\to\infty}\l p_k(\cdot), \overline{w_z^N(\cdot,t)}\ro= \lim_{N\to\infty}\sum_{n=1}^\infty c_{z,n}^N p_{k,n} (1-e^{-\lambda_nt})=\sum_{n=1}^\infty c_{z,n} p_{k,n} (1-e^{-\lambda_nt}).
 \end{equation*}
 
Now we claim that for $t>0$,  
\begin{equation}\label{claim}
 -\int_0^t \frac{\partial u}{\n}(z, \tau)\ d\tau
 =\int_0^t \sum_{k=1}^K \chi_{{}_{t-\tau\in [t_{k-1},t_k)}} 
 \Big[\sum_{n=1}^\infty c_{z,n}p_{k,n} (1-e^{-\lambda_n\tau})\Big]\ d\tau.
 \end{equation}
 If $t\le t_0$, from Lemma \ref{data_1} we have $\frac{\partial u}{\n}(z, \tau)=0$ for $\tau\le t$. This means \eqref{claim} is valid. 
 
For $t\in (t_{m_0-1},t_{m_0}]$ with $1\le m_0\le K$, Lemma \ref{data_1} yields that 
\begin{align*}
 -\int_0^t \frac{\partial u}{\n}(z, \tau)\ d\tau
 =&-\sum_{m=1}^{m_0-1}\int_{t_{m-1}}^{t_m} \frac{\partial u}{\n}(z, \tau)\ d\tau
 -\int_{t_{m_0-1}}^t \frac{\partial u}{\n}(z, \tau)\ d\tau\\
 =&\sum_{m=1}^{m_0-1}\int_{t_{m-1}}^{t_m} \sum_{l=1}^{m-1}\sum_{n=1}^\infty c_{z,n} p_{l,n} (e^{-\lambda_n(\tau-t_l)}-e^{-\lambda_n(\tau-t_{l-1})})\ d\tau\\
 &+\sum_{m=1}^{m_0-1}\int_{t_{m-1}}^{t_m}\sum_{n=1}^\infty c_{z,n} p_{m,n} (1-e^{-\lambda_n(\tau-t_{m-1})}) \ d\tau\\
 &+\int_{t_{m_0-1}}^t \sum_{l=1}^{m_0-1}\sum_{n=1}^\infty c_{z,n} p_{l,n} (e^{-\lambda_n(\tau-t_l)}-e^{-\lambda_n(\tau-t_{l-1})})\ d\tau\\
 &+\int_{t_{m_0-1}}^t \sum_{n=1}^\infty c_{z,n} p_{m_0,n} (1-e^{-\lambda_n(\tau-t_{m_0-1})}) \ d\tau\\
 =:&S_1+S_2+S_3+S_4.
\end{align*}
The straightforward computation gives that 
\begin{align*}
 S_1=&\sum_{m=1}^{m_0-1}\sum_{l=1}^{m-1}\int_{t_{m-1}-t_{l-1}}^{t_m-t_{l-1}} \sum_{n=1}^\infty c_{z,n} p_{l,n} (1-e^{-\lambda_n\tau})\ d\tau-\sum_{m=1}^{m_0-1}\sum_{l=1}^{m-1}\int_{t_{m-1}-t_l}^{t_m-t_l} \sum_{n=1}^\infty c_{z,n} p_{l,n} (1-e^{-\lambda_n\tau})\ d\tau\\
 =&\sum_{l=1}^{m_0-2}\sum_{m=l+1}^{m_0-1}\int_{t_{m-1}-t_{l-1}}^{t_m-t_{l-1}} \sum_{n=1}^\infty c_{z,n} p_{l,n} (1-e^{-\lambda_n\tau})\ d\tau
 -\sum_{l=1}^{m_0-2}\sum_{m=l+1}^{m_0-1}\int_{t_{m-1}-t_l}^{t_m-t_l} \sum_{n=1}^\infty c_{z,n} p_{l,n} (1-e^{-\lambda_n\tau})\ d\tau\\
 =&\sum_{l=1}^{m_0-2}\int_{t_l-t_{l-1}}^{t_{m_0-1}-t_{l-1}} \sum_{n=1}^\infty c_{z,n} p_{l,n} (1-e^{-\lambda_n\tau})\ d\tau
 -\sum_{l=1}^{m_0-2}\int_0^{t_{m_0-1}-t_l} \sum_{n=1}^\infty c_{z,n} p_{l,n} (1-e^{-\lambda_n\tau})\ d\tau.
\end{align*}
Similarly, we have 
\begin{align*}
 S_2=&\sum_{m=1}^{m_0-1}\int_0^{t_m-t_{m-1}}\sum_{n=1}^\infty c_{z,n} p_{m,n} (1-e^{-\lambda_n\tau}) \ d\tau,\\
 S_3=&\sum_{l=1}^{m_0-1}\int_{t_{m_0-1}-t_{l-1}}^{t-t_{l-1}} \sum_{n=1}^\infty c_{z,n} p_{l,n} (1-e^{-\lambda_n\tau})\ d\tau
 -\sum_{l=1}^{m_0-1}\int_{t_{m_0-1}-t_l}^{t-t_l} \sum_{n=1}^\infty c_{z,n} p_{l,n} (1-e^{-\lambda_n\tau})\ d\tau,\\
 S_4=&\int_0^{t-t_{m_0-1}} \sum_{n=1}^\infty c_{z,n} p_{m_0,n} (1-e^{-\lambda_n\tau}) \ d\tau=\int_0^t \chi_{{}_{t-\tau\in [t_{m_0-1},t_{m_0})}} \sum_{n=1}^\infty c_{z,n} p_{m_0,n} (1-e^{-\lambda_n\tau}) \ d\tau.
\end{align*}
These give that 
\begin{align*}
 S_1+S_2+S_3=&\sum_{l=1}^{m_0-2}\int_0^{t-t_{l-1}} \sum_{n=1}^\infty c_{z,n} p_{l,n} (1-e^{-\lambda_n\tau})\ d\tau
 -\sum_{l=1}^{m_0-2}\int_0^{t-t_l} \sum_{n=1}^\infty c_{z,n} p_{l,n} (1-e^{-\lambda_n\tau})\ d\tau\\
 &+\int_0^{t-t_{m_0-2}} \sum_{n=1}^\infty c_{z,n} p_{m_0-1,n} (1-e^{-\lambda_n\tau}) \ d\tau
 -\int_0^{t-t_{m_0-1}} \sum_{n=1}^\infty c_{z,n} p_{m_0-1,n} (1-e^{-\lambda_n\tau}) \ d\tau\\
 =&\sum_{l=1}^{m_0-1}\int_{t-t_l}^{t-t_{l-1}} \sum_{n=1}^\infty c_{z,n} p_{l,n} (1-e^{-\lambda_n\tau})\ d\tau\\
 =&\int_0^t \sum_{l=1}^{m_0-1} \chi_{{}_{t-\tau\in [t_{l-1},t_l)}} 
 \Big[\sum_{n=1}^\infty c_{z,n}p_{l,n} (1-e^{-\lambda_n\tau})\Big]\ d\tau. 
\end{align*}
Hence, it holds that 
\begin{align*}
  -\int_0^t \frac{\partial u}{\n}(z, \tau)\ d\tau&=S_1+S_2+S_3+S_4\\
  &=\int_0^t \sum_{l=1}^{m_0} \chi_{{}_{t-\tau\in [t_{l-1},t_l)}} 
 \Big[\sum_{n=1}^\infty c_{z,n}p_{l,n} (1-e^{-\lambda_n\tau})\Big]\ d\tau,
\end{align*}
which confirms the claim. 

When $K$ is finite, we can pick $t>0$ such that $t>t_K$. 
In this case, the term $S_4$ is eliminated, and $m_0-1$ is replaced by $K$ in the  
terms $S_1,S_2,S_3$. So the claim \eqref{claim} is still valid. 

Above all, we have proved the claim \eqref{claim} and the proof is complete. 
\end{proof}

\subsection{The analysis of Laplace transform.}\label{subsection_laplace}
The proof of Theorem \ref{uniqueness} will rely on the Laplace transform $\L$, defined as 
\begin{equation*}
 \L\{\psi(t)\}(s) =\int_0^\infty e^{-st} \psi(t)\ dt,
 \ s\in\mathbb C.
\end{equation*} 
Recalling the result in Corollary \ref{data_formula}, it is not hard to see that 
\begin{equation*}
 \L \{1-e^{-\lambda_nt}\}(s)=\lambda_n s^{-1}(s+\lambda_n)^{-1},\ \re s>0.
\end{equation*}
Also, Assumption \ref{condition_f_regularity} gives that 
\begin{align*}
 \Big|\int_0^t \sum_{k=1}^K \chi_{{}_{t-\tau\in [t_{k-1},t_k)}}\Big[\sum_{n=1}^\infty c_{z,n}p_{k,n} (1-e^{-\lambda_n\tau})\Big]\ d\tau\Big|\le 2t\sum_{k=1}^K \sum_{n=1}^\infty |c_{z,n}p_{k,n}|\le Ct,
\end{align*}
and we have $|e^{-st}t|$ is integrable on $(0,\infty)$ for $\re s>0$. These mean the Dominated Convergence Theorem can be used and it gives that for $\re s>0,$ 
\begin{equation*}
 \begin{aligned}
  \L\Big\{-\int_0^t \frac{\partial u}{\n}(z, \tau)\ d\tau\Big\}(s)
 &=\sum_{k=1}^K \int_0^\infty e^{-st}\int_0^t  \chi_{{}_{t-\tau\in [t_{k-1},t_k)}} 
 \Big[\sum_{n=1}^\infty c_{z,n}p_{k,n}(1-e^{-\lambda_n\tau})\Big]\ d\tau\ dt\\
 &=\sum_{k=1}^K \Big[\int_0^\infty e^{-st}\chi_{{}_{t\in [t_{k-1},t_k)}}\ dt\Big] 
 \ \Big[\int_0^\infty e^{-st}\sum_{n=1}^\infty c_{z,n}p_{k,n} 
 (1-e^{-\lambda_nt})\ dt\Big]\\
 &=\sum_{k=1}^K s^{-2}(e^{-t_{k-1}s}-e^{-t_ks})
 \ \Big[\sum_{n=1}^\infty c_{z,n}p_{k,n}\lambda_n(s+\lambda_n)^{-1}\Big].
 \end{aligned}
\end{equation*}
Hence, we see that for $\re s>0$, 
\begin{equation}\label{laplace}
s^2\L\Big\{-\int_0^t \frac{\partial u}{\n}(z, \tau)\ d\tau\Big\}(s)
 =\sum_{k=1}^K (e^{-t_{k-1}s}-e^{-t_ks})\ 
 \Big[\sum_{n=1}^\infty c_{z,n}p_{k,n}\lambda_n(s+\lambda_n)^{-1}\Big].
\end{equation}

We index the set of distinct eigenvalues as $\{\lambda_j\}_{j=1}^\infty$ with increasing order. Then we give the next lemma, which contains the well-definedness and the analyticity of the above complex-valued series.  
\begin{lemma}\label{analytic}
 Under Assumption \ref{condition_f_regularity}, we have the following properties:  
 \begin{itemize}   
 \item [(a)] For $R>0$ we define $\Lambda_R:=\{s\in\mathbb C:|s|<R\}$, then for each $k\in\{1,\cdots,K\}$,\\ $\sum_{n=1}^\infty c_{z,n}p_{k,n} \lambda_n(s+\lambda_n)^{-1}$ is uniformly convergent for $s\in \Lambda_R\setminus \{-\lambda_j\}_{j=1}^\infty$; 
  \item [(b)] $\sum_{n=1}^\infty c_{z,n}p_{k,n} \lambda_n(s+\lambda_n)^{-1}$ is holomorphic on $\mathbb C\setminus\{-\lambda_j\}_{j=1}^\infty$; 
 \item [(c)] The series $\sum_{k=1}^K (e^{-t_{k-1}s}-e^{-t_ks})\ \big[\sum_{n=1}^\infty c_{z,n}p_{k,n}\lambda_n(s+\lambda_n)^{-1}\big]$ is analytic on\\ $\mathbb C^+:=\{s\in\mathbb C:\re s>0\}$.
 \end{itemize}
\end{lemma}
\begin{proof}
For $(a)$, fixing $R>0$, since $\lim_{n\to\infty} \lambda_n=\infty$, then we can find large enough $N_1\in\mathbb N^+$ such that $\lambda_n\ge 2R$ if $n>N_1$. For $s\in \Lambda_R\setminus \{-\lambda_j\}_{j=1}^\infty$ and $n>N_1$, 
$$|s+\lambda_n|\ge |\re s+\lambda_n|=\re s+\lambda_n\ge \lambda_n-R.$$
This means that $|\lambda_n(s+\lambda_n)^{-1}|\le 2$. Given $\epsilon>0$, from Assumption \ref{condition_f_regularity}, we can find $N_2>0$ such that $\sum_{n=n_0}^\infty |c_{z,n}p_{k,n}|<\epsilon/2$ if $n_0>N_2$. Hence, for $n_0>\max\{N_1,N_2\}$, we have
\begin{equation*}
 \Big|\sum_{n=n_0}^\infty c_{z,n}p_{k,n} \lambda_n(s+\lambda_n)^{-1}\Big|\le 2\sum_{n=n_0}^\infty |c_{z,n}p_{k,n}|<\epsilon\ \text{for}\ s\in \Lambda_R\setminus\{-\lambda_j\}_{j=1}^\infty.
\end{equation*}
This gives the uniform convergence. \\

For $(b)$, fixing $R>0$, it is clear that $c_{z,n}p_{k,n}\lambda_n(s+\lambda_n)^{-1}$ is holomorphic on $\Lambda_R\setminus\{-\lambda_j\}_{j=1}^\infty$. Then the uniform convergence gives that 
$\sum_{n=1}^\infty c_{z,n}p_{k,n}\lambda_n(s+\lambda_n)^{-1}$ is 
holomorphic on $\Lambda_R\setminus\{-\lambda_j\}_{j=1}^\infty$. For each $s_0\in\mathbb C\setminus \{-\lambda_j\}_{j=1}^\infty$, we can find a sufficiently large $R>0$ such that $s_0\in \Lambda_R\setminus\{-\lambda_j\}_{j=1}^\infty$. This means that $\sum_{n=1}^\infty c_{z,n}p_{k,n}\lambda_n(s+\lambda_n)^{-1}$ is holomorphic on $s=s_0$, which leads to the desired result.\\

For $(c)$, for $s\in\mathbb C^+$, we have $|e^{-t_{k-1}s}-e^{-t_ks}|\le2$ and 
$|\lambda_n(s+\lambda_n)^{-1}|\le 1$. Also Assumption \ref{condition_f_regularity} gives that $\sum_{k=1}^K \sum_{n=1}^\infty |c_{z,n}p_{k,n}|<\infty$, then following the proof for $(a)$, we obtain that the series 
$$\sum_{k=1}^K (e^{-t_{k-1}s}-e^{-t_ks})\ 
 \Big[\sum_{n=1}^\infty c_{z,n}p_{k,n}\lambda_n(s+\lambda_n)^{-1}\Big]$$ 
 is uniformly convergent on $\mathbb C^+$. Sequentially, from the proof for $(b)$, the holomorphic result can be derived and the proof is complete.  
\end{proof}

\subsection{Some auxiliary lemmas.}
Before to show Theorem \ref{uniqueness}, we list and prove several auxiliary lemmas. 
\begin{lemma}\label{lemma_uniqueness_1}
Recall that $\{\lambda_j\}_{j=1}^\infty$ is the set of distinct eigenvalues with increasing order. 
For any nonempty open subset $\Gamma\subset\partial\Omega$, if $$\sum_{\lambda_n=\lambda_j} c_{z,n}\eta_n=0\ \text{for}\  j\in\mathbb{N}^+\ \text{and\ a.e.}\  z\in\Gamma,$$ 
then $\{\eta_n\}_{n=1}^\infty=\{0\}$. 
\end{lemma}
\begin{proof}
Fixing $j\in \mathbb N^+$, from the proof of Lemma \ref{c_z,n}, we have  
 \begin{equation*}
 \begin{aligned}
  c_{z,n}&=\lambda_j^{-1}\lim_{N\to \infty}  \l \psi_z^N,\A\overline{\varphi_n}\ro\\
  &=\lambda_j^{-1}\lim_{N\to \infty}\Big(\l \A\psi_z^N,\overline{\varphi_n}\ro-\l\psi_z^N,\frac{\partial\overline{\varphi_n}}{\n}\rp+\l\frac{\partial\psi_z^N}{\n},\overline{\varphi_n}\rp\Big)\\
  &=-\lambda_j^{-1}\sum_{l=1}^\infty \tilde\xi_l(z)\l\frac{\partial\varphi_n}{\n},\tilde\xi_l\rp.
  \end{aligned}
 \end{equation*}
With the definition of the orthonormal basis $\{\tilde \xi_l\}_{l=1}^\infty$, the above series is actually finite. This gives that for a.e. $z\in\Gamma$, 
\begin{equation*}
 \sum_{l=1}^\infty \tilde\xi_l(z)\l \frac{\partial\varphi_n}{\n},\tilde\xi_l\rp=\frac{\partial\varphi_n}{\n}(z),
\end{equation*}
which leads to $\sum_{\lambda_n=\lambda_j}\eta_n\frac{\partial\varphi_n}{\n}(z)=0$ for a.e. $z\in\Gamma$. Assume that $\{\eta_n:\lambda_n=\lambda_j\}\ne\{0\}$, then the linear independence of the set $\{\varphi_n(x):\lambda_n=\lambda_j\}$ yields that $\sum_{\lambda_n=\lambda_j}\eta_n\varphi_n(x)$ is not vanishing on $\Omega$. However, we see that $\sum_{\lambda_n=\lambda_j}\eta_n\varphi_n(x)$ is an eigenfunction corresponding to the eigenvalue $\lambda_j$. These and Lemma \ref{nonempty_open} give that $\sum_{\lambda_n=\lambda_j}\eta_n\frac{\partial\varphi_n}{\n}$ can not vanish almost everywhere on $\Gamma$, which is a contradiction. Hence, we prove that $\{\eta_n:\lambda_n=\lambda_j\}=\{0\}$ for each $j\in\mathbb N^+$, namely, $\eta_n=0$ for $n\in\mathbb N^+$. The proof is complete. 
\end{proof}

\begin{lemma}\label{lemma_uniqueness_2}
Set $\Gamma\subset\partial\Omega$ be a nonempty open subset and let $\sum_{n=1}^\infty c_{z,n} \eta_n$ be absolute convergent for a.e. $z\in \Gamma$. Then given $\epsilon>0$, the result     
$$\sum_{n=1}^\infty c_{z,n} \eta_n(1-e^{-\lambda_nt})=0\ \text{for}\ t\in(0,\epsilon)\ \text{and\ a.e.}\ z\in \Gamma$$ 
implies that $\eta_n=0$ for $n\in\mathbb N^+$. 
\end{lemma} 
\begin{proof}
From \cite[Lemma 3.5]{RundellZhang:2020} we have that $\sum_{\lambda_n=\lambda_j} c_{z,n} \eta_n=0$ for each $j\in\mathbb N^+$ and a.e. $z\in\Gamma$. Then Lemma \ref{lemma_uniqueness_1} gives the desired result and completes the proof. 
\end{proof}

\begin{lemma}\label{lemma_uniqueness_3}
 Let the conditions in Lemma \ref{lemma_uniqueness_2} be valid. 
 For $s\in \mathbb C^+$, if $\exists\epsilon>0$ such that 
 \begin{equation*}
  \lim_{\re s\to \infty}e^{\epsilon s} \Big[\sum_{n=1}^\infty c_{z,n}
  \lambda_n\eta_n(s+\lambda_n)^{-1}\Big]=0\ \text{for a.e.}\ z\in \Gamma, 
 \end{equation*} 
 then $\{\eta_n\}_{n=1}^\infty=\{0\}$.
\end{lemma}
\begin{proof}
For a.e. $z\in \Gamma,$ we define  
\begin{equation*}
F_z(t):=\chi_{{}_{t\ge 0}}\sum_{n=1}^\infty c_{z,n}\eta_n(1-e^{-\lambda_nt}),\ t\in(-\infty,\infty).
\end{equation*}

From the absolute convergence of the series $\sum_{n=1}^\infty c_{z,n}\eta_n$, we see that   
$$|F_z(t)|\le C\sum_{n=1}^\infty |c_{z,n}\eta_n|\le C<\infty.$$ 
Also, the direct computation yields that 
\begin{equation*}
\begin{aligned}
 \int_{-\epsilon}^\infty |e^{-st}|\int_0^{t+\epsilon}\ |F_z(\tau)|\ d\tau\ dt 
 &\le C \int_{-\epsilon}^\infty e^{-t\re s}(t+\epsilon)\ dt\\
&=C e^{\epsilon \re s}\int_0^\infty e^{-t\re s}t\ dt\\
&=C e^{\epsilon \re s} (\re s)^{-2} <\infty.
\end{aligned}
\end{equation*}
This implies the well-definedness of the integration $\int_{-\epsilon}^\infty |e^{-st}|\int_0^{t+\epsilon}\ |F_z(\tau)|\ d\tau\ dt$ for $\re s>0$ and a.e. $z\in \Gamma$. 

Introducing the Heaviside function $H(t):=\chi_{{}_{t\ge 0}}$, we see that 
$$\int_{-\epsilon}^\infty e^{-st}\int_0^{t+\epsilon}\  
F_z(\tau)\ d\tau\ dt
=\int_{-\infty}^\infty e^{-st} \int_{-\infty}^\infty H(t-\tau+\epsilon) 
F_z(\tau)\ d\tau \ dt.$$ 
The above argument provides the well-definedness of the integral 
$$\int_{-\infty}^\infty e^{-st} \int_{-\infty}^\infty H(t-\tau+\epsilon) 
F_z(\tau)\ d\tau \ dt$$ when $\re s>0$. With Fubini Theorem and the Dominated Convergence Theorem, we have  
\begin{align*}
\int_{-\infty}^\infty e^{-ts} \int_{-\infty}^\infty H(t-\tau+\epsilon) 
F_z(\tau)\ d\tau \ dt&=\int_{-\epsilon}^\infty\!e^{-ts}\,dt\ \int_0^\infty e^{-\tau s} \Big[\sum_{n=1}^\infty c_{z,n}\eta_n(1-e^{-\lambda_n\tau})\Big]\,d\tau\\
&=s^{-2}e^{\epsilon s} \Big[\sum_{n=1}^\infty c_{z,n}\lambda_n\eta_n 
(s+\lambda_n)^{-1}\Big],\quad \re s>0,
\end{align*}
which leads to  
\begin{equation*}
 \lim_{\re s\to \infty}\int_{-\infty}^\infty e^{-ts} \int_{-\infty}^\infty H(t-\tau+\epsilon) F_z(\tau)\ d\tau \ dt=0.
\end{equation*}
Also, from the direct calculation, we have  
\begin{align*}
 \int_{-\infty}^\infty e^{-ts}\int_{-\infty}^\infty H(t-\tau+\epsilon) 
F_z(\tau)\ d\tau \ dt&=\int_{-\epsilon}^\infty e^{-ts} \Big[\sum_{n=1}^\infty c_{z,n}\int_0^{t+\epsilon}\eta_n(1-e^{-\lambda_n\tau})\ d\tau\Big] \ dt\\
&=\int_0^\epsilon e^{(\epsilon-t)s}\ \Big[\sum_{n=1}^\infty c_{z,n}\int_0^t\eta_n(1-e^{-\lambda_n\tau})\ d\tau\Big]\ dt\\
&\quad+\int_0^\infty e^{-ts}\ \Big[\sum_{n=1}^\infty c_{z,n}\int_0^{t+\epsilon}\eta_n(1-e^{-\lambda_n\tau})\ d\tau\Big]\ dt\\
&=:S_{z,1}(s)+S_{z,2}(s).
\end{align*}
For $S_{z,2}(s)$, the Dominated Convergence Theorem and the absolute convergence of $\sum_{n=1}^\infty c_{z,n}\eta_n$ give that   
\begin{equation*}
  |S_{z,2}(s)|\le C\sum_{n=1}^\infty |c_{z,n}\eta_n|\int_0^\infty e^{-t \re s} |t+\epsilon|\ dt,
\end{equation*} 
and clearly the right side converges to $0$ as $\re s\to \infty$. So $\lim_{\re s\to\infty}S_{z,2}(s)=0$.
With the limit assumption in this lemma, we have 
\begin{equation*}
 \lim_{\re s\to \infty} S_{z,1}(s) 
 =\lim_{\re s\to \infty} [S_{z,1}(s)+S_{z,2}(s)]-\lim_{\re s\to \infty} S_{z,2}(s)=0.
\end{equation*}
For each $R_0\in \mathbb R$, when $\re s\in (-\infty, R_0)$, we can see  
\begin{equation*}
 |S_{z,1}(s)|\le C\int_0^\epsilon e^{(\epsilon-t)\re s}t\ \sum_{n=1}^\infty 
 |c_{z,n}\eta_n|\ dt\le Ce^{\epsilon R_0}\int_0^\epsilon t\ dt=C_{\epsilon,R_0}<\infty. 
\end{equation*}
This with the limit result $\lim_{\re s\to \infty} S_{z,1}(s)=0$ yields  that $S_{z,1}(s)$ is well-defined and bounded on the whole complex plane $\mathbb{C}$. The definition of $S_{z,1}(s)$ gives that $S_{z,1}(s)$ is holomorphic on $\mathbb{C}$. Hence, we have that $S_{z,1}(s)$ is a bounded entire function for a.e. $z\in \Gamma$, which with Liouville's theorem yields that $S_{z,1}\equiv C$ on $\mathbb{C}$. Considering the limit result $\lim_{\re s\to \infty} S_{z,1}(s)=0$, it gives that 
\begin{equation*}
 S_{z,1}(s)=0\ \text{for}\ s\in\mathbb{C}\ \text{and a.e.}\ z\in \Gamma. 
\end{equation*}
Then we can see  
\begin{equation*}
\begin{aligned}
 S_{z,1}(s)=&\int_0^\epsilon e^{(\epsilon-t)s}\ \Big[\sum_{n=1}^\infty c_{z,n}\int_0^t \eta_n(1-e^{-\lambda_n\tau})\ d\tau\Big]\ dt\\
=&e^{\epsilon s}\int_0^\epsilon e^{-ts}\ \Big[\sum_{n=1}^\infty c_{z,n}\int_0^t \eta_n(1-e^{-\lambda_n\tau})\ d\tau\Big]\ dt\equiv 0,
\end{aligned}
\end{equation*}
which means for $\re s>0$, 
\begin{equation*}
\begin{aligned}
 0\equiv&\int_0^\infty e^{-st}H(\epsilon-t)\ \Big[\sum_{n=1}^\infty c_{z,n}\int_0^t \eta_n(1-e^{-\lambda_n\tau})\ d\tau\Big]\,dt \\
=&\L \Big\{H(\epsilon-t)\sum_{n=1}^\infty c_{z,n}\int_0^t \eta_n(1-e^{-\lambda_n\tau})\ d\tau\Big\}(s).
\end{aligned}
\end{equation*}
Using \cite[Corollary 8.1]{Folland:1992}, it follows that  
\begin{equation*}
 \sum_{n=1}^\infty c_{z,n}\int_0^t \eta_n(1-e^{-\lambda_n\tau})\ d\tau = 0,\ t\in(0,\epsilon).
\end{equation*} 
The absolute convergence of $\sum_{n=1}^\infty c_{z,n} \eta_n$ supports the uniform convergence of 
$$\sum_{n=1}^\infty c_{z,n} \int_0^t\eta_n(1-e^{-\lambda_n\tau})\ d\tau\ \text{and}\ \sum_{n=1}^\infty c_{z,n}\eta_n (1-e^{-\lambda_n t})$$ on $(0,\epsilon)$. Thus, from \cite[Theorem 7.17]{Rudin:1976}, differentiating the series   
$\sum_{n=1}^\infty c_{z,n} \int_0^t\eta_n(1-e^{-\lambda_n\tau})\ d\tau$ on $t$ yields that for $t\in(0,\epsilon)$ and a.e. $z\in \Gamma$, 
$$0=\frac{d}{dt}\Big[\sum_{n=1}^\infty c_{z,n} \int_0^t\eta_n(1-e^{-\lambda_n\tau})\ d\tau\Big] =\sum_{n=1}^\infty c_{z,n}\eta_n (1-e^{-\lambda_n t}).$$
This together with Lemma \ref{lemma_uniqueness_2} leads to $\eta_n=0$ for $n\in\mathbb N^+$, and completes the proof. 
\end{proof}

\subsection{Proof of Theorem \ref{uniqueness}.}

Now, it is time to show Theorem \ref{uniqueness}. For shorten the proof, the following notations will be used: 
\begin{equation*}
 \begin{aligned}
 &p_{k,n}=\l p_k(\cdot),\varphi_n(\cdot)\ro, 
 && Q_{z,k}(s)= \sum_{n=1}^\infty c_{z,n}p_{k,n}\lambda_n(s+\lambda_n)^{-1},\\
 &\tilde p_{k,n}=\l \tilde{p}_k(\cdot),\varphi_n(\cdot)\ro,
  &&\tilde Q_{z,k}(s)= \sum_{n=1}^\infty c_{z,n}\tilde p_{k,n}\lambda_n(s+\lambda_n)^{-1}.
 \end{aligned}
\end{equation*}

\begin{proof}[Proof of Theorem \ref{uniqueness}]
  The result \eqref{laplace}, Lemma \ref{analytic} and the analytic continuation give that for $s\in\mathbb{C}^+$ and a.e. $z\in\Gamma$,    
 \begin{equation}\label{equality_2}
    \sum_{k=1}^K (e^{-t_{k-1}s}-e^{-t_ks})Q_{z,k}(s)
 =\sum_{k=1}^{\tilde K} (e^{-\tilde t_{k-1}s}-e^{-\tilde t_ks})\tilde Q_{z,k}(s).
  \end{equation}

Firstly let us show $t_0=\tilde t_0$. Without loss of generality, we can assume that $t_0<\tilde t_0$. Picking $\epsilon>0$ such that $\epsilon<\min \{\tilde t_0- t_0,t_1-t_0\}$, from \eqref{equality_2}, we have for $s\in\mathbb{C}^+$ and a.e. $z\in\Gamma$, 
\begin{equation}\label{equality_3}
\begin{aligned}
  e^{\epsilon s} Q_{z,1}(s) =&e^{(\epsilon+t_0-t_1)s}Q_{z,1}(s)-\sum_{k=2}^K (e^{(\epsilon+t_0-t_{k-1})s}-e^{(\epsilon+t_0-t_k)s})Q_{z,k}(s)\\ 
 &+\sum_{k=1}^{\tilde K} (e^{(\epsilon+t_0- \tilde t_{k-1})s}-e^{(\epsilon+t_0-\tilde t_k)s})\tilde Q_{z,k}(s). 
\end{aligned}
\end{equation}
We see that $\lambda_n(s+\lambda_n)^{-1}$ is uniformly bounded on $\mathbb C^+$. Recalling the convergence result $\sum_{k=1}^K \sum_{n=1}^\infty |c_{z,n}p_{k,n}|<\infty$ from Assumption \ref{condition_f_regularity}, the uniform convergence of the series\\ $\sum_{k=2}^K (e^{(\epsilon+t_0-t_{k-1})s}-e^{(\epsilon+t_0-t_k)s})Q_{z,k}(s)$ on the half plane $\mathbb C^+$ can be derived. Then we have 
\begin{align*}
\lim_{\re s\to \infty} \sum_{k=2}^K (e^{(\epsilon+t_0-t_{k-1})s}
-e^{(\epsilon+t_0-t_k)s})Q_{z,k}(s)&= \sum_{k=2}^K\lim_{\re s\to \infty} (e^{(\epsilon+t_0-t_{k-1})s}
-e^{(\epsilon+t_0-t_k)s})  Q_{z,k}(s)\\
&=0.
\end{align*}
Similarly, we see that other terms in the right side of 
\eqref{equality_3} also tend to zero as $\re s \to \infty$. So we have $\lim_{\re s\to \infty}e^{\epsilon s} Q_{z,1}(s)=0$ for a.e. $z\in \Gamma.$ This with Lemma \ref{lemma_uniqueness_3} gives that $p_{1,n}=0$ for $n\in \mathbb N^+$, i.e. $\|p_1\|_{L^2(\Omega)}=0$, which contradicts with Assumption \ref{condition_f}. Hence, we have $t_0=\tilde t_0$.

Inserting $t_0=\tilde t_0$ into \eqref{equality_2} and picking $0<\epsilon<\min\{t_1-t_0,\tilde t_1-t_0\},$ we have  
\begin{equation*}
\begin{aligned}
  e^{\epsilon s} [Q_{z,1}(s)- \tilde Q_{z,1}(s)]  =&e^{(\epsilon+t_0-t_1)s}Q_{z,1}(s) -e^{(\epsilon+t_0-\tilde t_1)s}\tilde Q_{z,1}(s)  \\
  &-\sum_{k=2}^K (e^{(\epsilon+t_0-t_{k-1})s}-e^{(\epsilon+t_0-t_k)s})Q_{z,k}(s)\\ 
 &+\sum_{k=2}^{\tilde K} (e^{(\epsilon+t_0-\tilde t_{k-1})s}
 -e^{(\epsilon+t_0-\tilde t_k)s})\tilde Q_{z,k}(s). 
\end{aligned}
\end{equation*}
From the previous arguments, we can prove 
\begin{equation*}
 \lim_{\re s\to \infty}e^{\epsilon s} [Q_{z,1}(s)- \tilde Q_{z,1}(s)]=0,\ \text{a.e.}\ z\in \Gamma.
\end{equation*}
Lemma \ref{lemma_uniqueness_3} gives that $p_{1,n}=\tilde p_{1,n}$ for $n\in \mathbb N^+$, namely $\|p_1-\tilde p_1\|_{L^2(\Omega)}=0$. 

Next we want to show $t_1=\tilde t_1$. Owing the results $p_{1,n}=\tilde p_{1,n}$ and $t_0=\tilde t_0$ into \eqref{equality_2}, it follows that for $s\in\mathbb C^+$ and a.e. $z\in \Gamma$, 
\begin{equation*}
\begin{aligned}
e^{-t_1s}[Q_{z,2}(s)-Q_{z,1}(s)]=&e^{-\tilde t_1s}[\tilde Q_{z,2}(s)-\tilde Q_{z,1}(s)]+e^{-t_2s}Q_{z,2}(s)-e^{-\tilde t_2s}\tilde Q_{z,2}(s)\\
&-\sum_{k=3}^K (e^{-t_{k-1}s}-e^{-t_ks})Q_{z,k}(s)
 +\sum_{k=3}^{\tilde K} (e^{-\tilde t_{k-1}s}-e^{-\tilde t_ks})\tilde Q_{z,k}(s).
\end{aligned}
\end{equation*}
Without loss of generality we assume that $t_1<\tilde t_1$ and pick $0<\epsilon<\min\{\tilde t_1-t_1,t_2-t_1,\tilde t_2-t_1\}$. Following the proof for $t_0=\tilde t_0$, we can show that $\|p_2-p_1\|_{L^2(\Omega)}=0$. This is a contradiction. Hence $t_1=\tilde t_1$. 

Now with the results $t_0=\tilde t_0$, $t_1=\tilde t_1$ and $\|p_1-\tilde p_1\|_{L^2(\Omega)}=0$, from \eqref{equality_2}, we have that for $\re s>0$ and a.e. $z\in \Gamma$, 
\begin{equation*}
   \sum_{k=2}^K (e^{-t_{k-1}s}-e^{-t_ks})Q_{z,k}(s)
 =\sum_{k=2}^{\tilde K} (e^{-\tilde t_{k-1}s}-e^{-\tilde t_ks})\tilde Q_{z,k}(s).
\end{equation*}
From the previous proof, we can derive that $t_2=\tilde t_2$ and $\|p_2-\tilde p_2\|_{L^2(\Omega)}=0$. Using this technique recursively, we can conclude that 
\begin{equation}\label{equality_4}
 t_0=\tilde t_0,\ t_k=\tilde t_k,\ \|p_k-\tilde p_k\|_{L^2(\Omega)}=0\ \text{for}
 \ k=1,\cdots,\min\{K,\tilde K\}.
\end{equation}

Finally we need to show $K=\tilde K$ (For the case of $K=\infty$ and $\tilde K=\infty$, we regard them as $K=\tilde K$). Without loss of generality assuming that $K<\tilde K$,  from \eqref{equality_2} and \eqref{equality_4}, we have for $\re s>0$ and a.e. $z\in \Gamma$, 
\begin{equation*}
 \sum_{k=K+1}^{\tilde K} (e^{-\tilde t_{k-1}s}-e^{-\tilde t_ks})\tilde Q_{z,k}(s)=0.
\end{equation*}
The former proof gives that $\|\tilde p_{K+1}\|_{L^2(\Omega)}=0$, which is a contradiction. Hence, we have that $K=\tilde K$, which together with \eqref{equality_4} leads to the desired result. The proof is complete.  
\end{proof}

\section{Numerical reconstructions.}
\label{numerical_sec}
In this section, we are going to conduct several numerical experiments which realize the inversion process described in Theorem \ref{uniqueness}.
Here we take the example of locating the leaking oil tanker in the ocean. To be more specific,
suppose an oil tanker is planning to travel from harbor A to harbor B and the planned route is known. The ship is leaking oil during the trip and the exact trajectory of the ship is unknown. The target of the engineers is to find the exact location where the ship leaks the oil during the trip. 
The physics of the oil diffusion in the ocean can be modelled as a parabolic equation \eqref{PDE} and the leaking oil is the heat source in the PDE.
By Theorem \ref{uniqueness}, one can find these positions given the fluxes from a nonempty open subset of the boundary.

The oil tanker problem can be formulated as estimating the state of the system (position of the ship) when a sequence of observations in time becomes available. The popular choice to solve the sequential prediction problem is the sequential Monte Carlo (SMC) method, which is also called the particle filter (PF). 
In this work, we are going to use one variant of the SMC as the approach to solving the inverse problem.
Before presenting the results of the numerical experiments, we will first review some preliminaries of the SMC algorithm. 

\subsection{Sequential Monte Carlo (SMC).}
Consider a discrete-time Markov process $\{x_t\}_{t = 0}^{n}$:
\begin{align}
    x_0\sim p(x_0) \text{ and } x_n|x_{n-1}\sim p(x_n|x_{n-1}),
    \label{numerical_bayesian_model}
\end{align}
where $\sim$ means distributed as, $p(x_0)$ is the prior of the initial state $x_0$ and $p(x'|x)$ is the given transition probability
prior from current state $x$ to the next state $x'$.
We are interested in estimating $\{x_t\}_{t = 0}^n$ which is also denoted as $x_{0:n}$, given the observations $\{y_t\}_{t = 0}^n$ at each time step.
In the oil tanker application, the state $x_t$ is the position of the tanker and $y_t$ is the flux on the boundary. 

The equation \eqref{numerical_bayesian_model} together with the likelihood $p(y_t|x_t)$ will define a Bayesian model. We have the prior on the trajectory $x_{0:n} : = \{x_0, ..., x_n\}$ as follow:
\begin{align}
    p(x_{0:n}) = p(x_0)\Pi_{k = 1}^n p(x_k|x_{k-1}).
    \label{numerical_prior}
\end{align}
Moreover,
\begin{align}
    p(y_{0:n}|x_{0:n}) = \Pi_{k = 1}^n p(y_k|x_k).
    \label{numerical_likelihood}
\end{align}
From a Bayesian perspective, the posterior distribution of $x_{0:n}$ given a sequence of the observations $y_{0:n}$ is:
\begin{align}
    p(x_{0:n}|y_{0:n}) = \frac{p(x_{0:n}, y_{0:n} )}{p(y_{0:n})},
    \label{numerical_bayesian}
\end{align}
where $p(x_{0:n}, y_{0:n}) = p(y_{0:n}|x_{0:n})  p(x_{0:n})$
and $p(y_{0:n}) = \int p(x_{0:n}, y_{0:n})dx_{0:n}$.

In our application, the most time-consuming step is to evaluate the likelihood function $p(y_{0:n}|x_{0:n})$, and we will follow the classical framework as in \cite{stuart2010inverse}.
Denoting the forward solver which maps the input state $u$ to the observation $y$ as $\mathcal{G}$, it follows that
\begin{align*}
    y = \mathcal{G}(u)+\eta,
\end{align*}
where $\eta$ is the error associated with the process. There are various sources of error. For example, the error in the observation and the error of the forward solver. The total error $\eta$ is then assumed to follow the Gaussian distribution $\mathcal{N}(0, B)$ where $B$ is the covariance with a proper size. The likelihood function then has the form:
\begin{align*}
    p(y|u) := p(y-\mathcal{G}(u)) = p(\eta)\sim \mathcal{N}(0, B).
\end{align*}

Now the area of interesting is the tracking problem: find the current state given the observations. Theoretically, this means that one needs to find a group of $p(x_{0:n}|y_{0:n})$. 
By the prior \eqref{numerical_prior} and the likelihood \eqref{numerical_likelihood}, the joint distribution $p(x_{0:n}, y_{0:n})$ in \eqref{numerical_bayesian} then satisfies
\begin{align*}
    p(x_{0:n}, y_{0:n}) = p(x_{1:n}, y_{1:n})p(y_n|x_n)p(x_n|x_{n-1}),
\end{align*}
and it follows that 
\begin{align}
    p(x_{0:n}|y_{0:n}) = p(x_{0:n-1}|y_{0:n-1})\frac{p(x_n|x_{n-1}) p(y_n|x_n)}{p(y_n|y_{n-1})},
    \label{numerical_post}
\end{align}
where $p(y_n|y_{n-1}) = \int p(x_{n-1}|y_{0:n-1})p(x_n|x_{n-1})p(y_n|x_n) dx_{{n-1:n}}$. 

Most particle filtering methods are created by a numerical approximation to \eqref{numerical_post}. A common and powerful algorithm is the sequential Monte Carlo (SMC) approximation.  One can use a set of samples (particles) which are drawn from a posterior distribution to approximate the target posterior.
More precisely, the posterior is the average calculated by
\begin{align*}
    \hat{p}(x_{0:n}|y_{0:n}) = \frac{1}{N_p}\sum_{i = 1}^{N_p}\delta_{x_{0:t}^i}dx_{0:t}.
\end{align*}
Here, $\{x_{0:t}^i\}_{i = 1}^{N_p}$ are the particles and $\delta(d\cdot)$ is the standard delta function.
As a result, when the number of samples $N_p$ is large enough, the average will approximate the true distribution.
One of the advantages of the method is that one can prove the asymptotic convergence to the target distribution of interest. Besides, compared to some other approaches such as the extended Kalman filter and unscented Kalman filter, the SMC approach does not assume that the state models are Gaussian \cite{doucet2009tutorial, van2000unscented}. One hence can apply the method in a broader area.

Unfortunately, it is not easy to draw samples from the posterior. An alternative way of drawing is to sample from a proposal distribution $q(x_{0:t}|y_{0:t})$ and to approximate the posterior using a weighted sum of particles. One usually calls this methodology as the Bayesian important sampling (IS). 

Now, two questions remain unsolved. One is how to design a proposal distribution; another one is how to calculate the weight for each sample.
How to select $q(x_{0:t}|y_{0:t})$ is an interesting research topic.
One usually requires that the $q(x_{0:t}|y_{0:t})$ satisfies the following structure:
\begin{align*}
    q(x_{0:t}|y_{0:t}) = q(x_{0:t-1}|y_{1:t-1})q(x_t|x_{0:t-1}, y_{1:t}).
\end{align*}
In this work, we are going to set $q(x_0|y_0) = p(x_0)$ and $q(x_t|x_{0:t-1}, y_{1:t}) = p(x_t|x_{t-1})$. 
According to \cite{gordon1993novel, andrieu2010particle}, if the latent variable dimension is not large and the observations are not too informative,
setting the proposal to be equal to transition probability is good enough.
To verify our method, we will conduct experiments with the transition probabilities $p(x_t|x_{t-1})$; the detailed design is shown in Section \ref{numerical_1} and Section \ref{numerical_2}. The weights can then be calculated as follows:
\begin{align*}
    w_t = w_{t-1}\frac{p(y_t|x_t)p(x_t|x_{t-1})}{q(x_t|x_{0:t-1}, y_{1:t})},
\end{align*}
and $w_1 = \frac{p(x_1)p(y_1|x_1)}{q(x_1|y_1)}$, where $p(y_t|x_t)$ is the likelihood. We omit the derivation and one can refer to \cite{doucet2009tutorial, van2000unscented, andrieu2010particle} for the details.

One issue of the algorithm is the degeneration of the weights, that is, the variance of the importance
weights will become larger and larger with respect to time. Consequently, there are only a few samples which have a meaningful weight when $t$ is large. One way to reduce the effect of degeneration is to use the resampling method. As a result, the algorithm is called the sequential importance resampling (SIR). We provide an SIR algorithm in the Appendix \ref{app_sir}. There are various of resampling methods;
in this work, we apply the systematic resampling as suggested in \cite{doucet2009tutorial}; please check Appendix \ref{app_sir} for the details.

\subsection{Sources.}
\label{sec_sources}
In the next two subsections, we are going to demonstrate how we set up the testing problems.
For equation \eqref{PDE}, we set $\Omega = [0, 1]^2$ and the terminal simulation time is $T = 0.1$.
Also $\kappa(x)$ is the given permeability which will be defined later in Section \ref{numerical_kappa_sec}, and 
the source $\sum_{k = 1}^{K}p_k(x)\chi_{{}_{t\in[t_{k-1}, t_k)}}$ is partially known. More specifically, the time mesh grid $\{t_k\}_{k=0}^K$ is given and the spatial components 
$\{p_k(x)\}_{k = 1}^K$ are set to be characteristic functions. The k-th source $p_k(x)$ is defined as:
\begin{equation}
  p_k(x) = \begin{cases}
            p_k,\quad x\in A_k, \\
            0,\quad\ \text{otherwise},
        \end{cases}
        \label{numerical_delta}
\end{equation}
Here $p_k$ is a known constant; $A_k$ is a square with a known size (is equal to $0.06\times 0.06$ in all our experiments), but the location is not given. Our target work is to find one vertex coordinate of the $A_k$ and $k = 1, ..., K$. One possible $p_k(x)$ is shown in Figure \ref{numerical_source} (left). In the following experiments, we will consider 2 sources. In both cases, $K = 10$ and $p_k = 1250$ for all $k$. 

We here list the vertices (left top) coordinates for both sources and we will trace these two sequences of the coordinates later by our proposed method. The first list of vertices is: 
(0.12, 0.12), (0.20, 0.20), (0.28, 0.28), (0.36, 0.36), (0.44, 0.44), (0.52, 0.52), (0.60, 0.60), (0.68, 0.68), (0.76, 0.76), (0.84, 0.84); 
the second list of vertices is:  
(0.12, 0.12), (0.20, 0.24), (0.28, 0.36), (0.36, 0.48), (0.44, 0.56), (0.52, 0.64), (0.60, 0.72), (0.68, 0.78), (0.76, 0.84), (0.84, 0.90).
For better illustration, we plot the trajectories of the two sources in Figure \ref{numerical_source} and we will discuss these two cases in Section \ref{numerical_1} and Section \ref{numerical_2} respectively. In the oil tanker application, the ship leaks oil in $A_k$ for $k = 1, 2..., K$ and we need to trace the position of all $A_k$ (the left top vertices) given some prior information about how this ship is moving (velocity of the ship).
\begin{figure}[H]
\centering
\includegraphics[scale = 0.35]{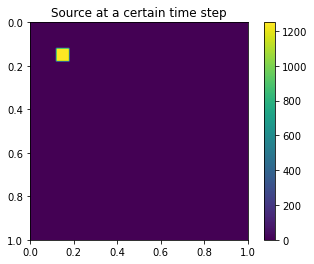}
\includegraphics[scale = 0.35]{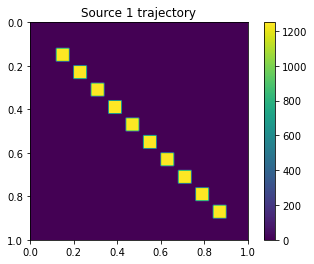}
\includegraphics[scale = 0.35]{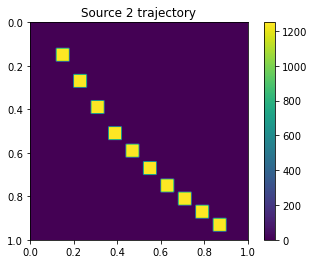}
\caption{Demonstration of the source. Left: example of $p_k(x)$ which is the characteristic function defined in equation \eqref{numerical_delta}. Middle and right images show the trajectories of the source $\{p_k(x)\}_{k = 1}^K$.   
Middle: source one, which is used in Section \ref{numerical_1}. Right: source two, which is used in Section \ref{numerical_2}.} \label{numerical_source}
\end{figure}

\subsection{Permeability fields.}
\label{numerical_kappa_sec}
In equation \eqref{PDE}, the operator $\A$ includes the permeability field $\kappa(x)$. 
In order to better demonstrate the idea of the theorem and the algorithm, we will solve several multiscale problems.
Two different permeability fields are used here.  
The first $\kappa_1(x)$ has multiple frequencies and is defined as follow:
\begin{align*}
    \kappa_1(x, y) = 15\sin(2\pi\cdot0.01x)\cdot\sin(2\pi\cdot 0.05y)+0.05\sin(2\pi\cdot6.5x)\cdot\sin(2\pi\cdot6y)+0.1,
    \label{numerical_multifreq}
\end{align*}
where $(x, y)\in\Omega: = [0, 1]^2$ and it is demonstrated in Figure \ref{kappa_1}.
\begin{figure}[H]
\centering
\includegraphics[scale = 0.45]{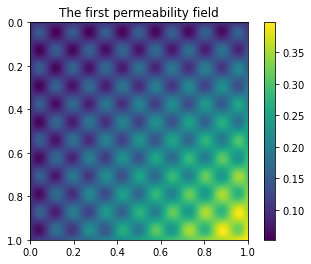}
\caption{The first permeability field 2D demonstration. $\max \kappa_1(x) = 0.3978$ and $\min \kappa_1(x) = 0.0521$.} 
\label{kappa_1}
\end{figure}
The second permeability field $\kappa_2(x)$ (see Figure \ref{numerical_kappa_2} for the illustration) has multiple high contrast channels and is widely used in the multiscale finite element method society \cite{chung2021computational, chetverushkin2021computational}. 
\begin{figure}[H]
\centering
\includegraphics[scale = 0.45]{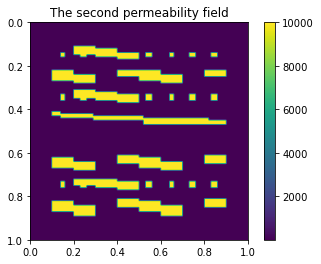}
\caption{The second permeability field 2D demonstration. The permeability in the yellow channels and dots is equal to $10^4$, while the permeability in the purple background is equal to $1$.} 
\label{numerical_kappa_2}
\end{figure}
The multiple frequencies in $\kappa_1(x)$ and high contrast in $\kappa_2(x)$ bring difficulty in solving the corresponding equations numerically.
In this work, we will use the standard finite element solver with spatial discretization $\Delta x = 10^{-2}$.
For the temporal approximation, the backward Euler scheme is adopted, and the time discretization is $\Delta t = 10^{-3}$.

\subsection{The first set of experiments.}
\label{numerical_1}
In this part, we present the first set of experiments whose source is defined in the middle image of Figure \ref{numerical_source}. As we have discussed before, the proposal distribution $q(x'|y', x)$ is the same with the  transition probability $p(x'|x)$ where $x', y'$ denote the new sample and the observation of the next time step, respectively. We hence only need to specify $p(x'|x)$, which is defined in Algorithm \ref{numerical_transition1}.
\begin{algorithm}[H]
\SetAlgoLined
Input: current $x = (x_1, x_2)$\;
Initialize: Set the size related parameters $r_1, r_2>0$ and speed related parameter $s_0$, and draw a random number $s\sim \mathcal{U}(0, s_t)$. We then define $l_1 = s_0+s+r_1$ and $l_2 = s_0+s+r_2$\;
Create a rectangle: the top right vertex coordinate is $(x_1+l_1\Delta x, x_2+l_2\Delta x)$ and the size is equal to $(2r_1\Delta x, 2r_2\Delta x)$\;
Draw a point $x'$ from the rectangle above uniformly and this point will be the next sample.
\caption{Transition algorithm}
\label{numerical_transition1}
\end{algorithm}
Here we set $s_0 = 6,\ s_t = 5$ and $r_1 = r_2 = 4$.
It should be noted that if one is given the exact current state $x = (x_1, x_2)$, the next exact state is $x' = (x_1 + 8\Delta x, x_2+8\Delta x)$ (please check the trajectory of the first source in Section \ref{sec_sources});  however, the center (mean) of all proposed regions is deviated from the $x'$. 
This simulates the real life scenario: one only knows the ship's route plan; however the ship may depart from the original route, and one needs to find the real route and positions when the ship leaks.

The observation is the flux of the boundary. For the multiple frequencies example, we measure the flux in the interval $[0.45, 0.52]$ from each boundary; while for the high contrast permeability example, we only measure the flux in the interval $[0.5, 0.55]$. In both cases, the number of particles is $320$. To evaluate our results, we calculate the mean of the samples at each time step $k$ (this indicates the time step in the source) and compare it with the pre-set values. 
The demonstration and the relative error are shown in Figure \ref{numerical_exp1_multi}.
\begin{figure}[h!]
\centering
\includegraphics[scale = 0.4]{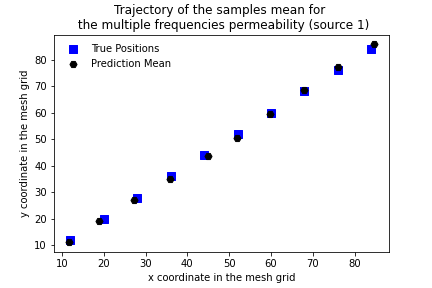}
\includegraphics[scale = 0.4]{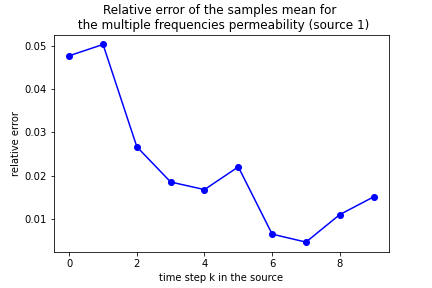}
\caption{Samples mean and the relative error of the multiple frequencies permeability with the first source. Left: samples mean and true source at each time step $k$.  Right: relative error of the samples mean.} 
\label{numerical_exp1_multi}
\end{figure}

\begin{figure}[h!]
\centering
\includegraphics[scale = 0.4]{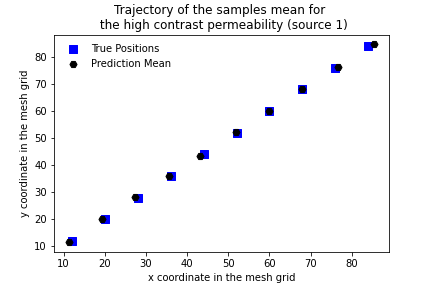}
\includegraphics[scale = 0.4]{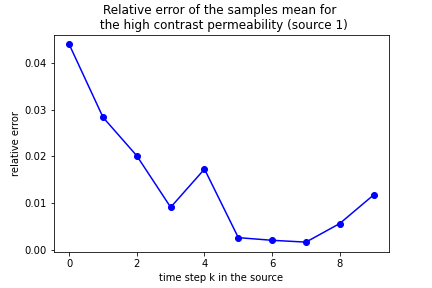}
\caption{Samples mean and the relative error of the high contrast permeability with the first source. Left: samples mean and true source at each time step $k$.  Right: relative error of the samples mean.} 
\label{}
\end{figure}

\subsection{The second set of experiments.}
\label{numerical_2}
In this set of experiments, we will use the second source which is demonstrated as the right image in Figure \ref{numerical_source}.
We use the same transition algorithm in Algorithm \ref{numerical_transition1} but we use a different set of parameters, specifically, $s_0 = 4,\ s_t = 5,\ r_1 = 4,\ r_2 = 5$. Similar to the first source, the center (mean) of all proposed regions is deviated from $x'$. This is a real situation in which the ship's actual route is different from the original plan. The setting makes the predictions more challenging. 
For the flux measurement, we choose an interval $[0.40, 0.60]$ from each boundary for the multiple frequencies permeability and $[0.43, 0.58]$ from each boundary for the high contract permeability. The results are shown in Figures  \ref{numerical_exp2_multi} and \ref{numerical_exp2_high},  respectively.
\begin{figure}[h!]
\centering
\includegraphics[scale = 0.4]{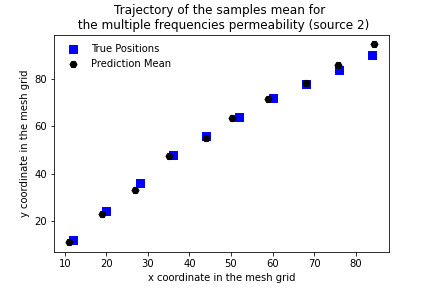}
\includegraphics[scale = 0.4]{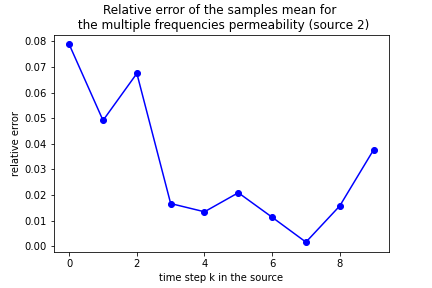}
\caption{Samples mean and the relative error of the multiple frequencies permeability with the second source. Left: samples mean and true source at each time step $k$.  Right: relative error of the samples mean.} 
\label{numerical_exp2_multi}
\end{figure}

\begin{figure}[h!]
\centering
\includegraphics[scale = 0.4]{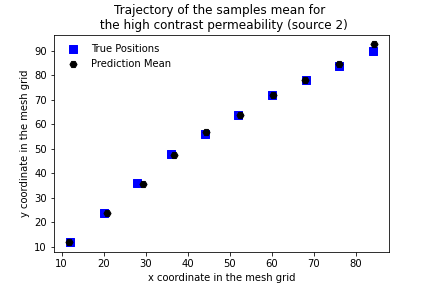}
\includegraphics[scale = 0.4]{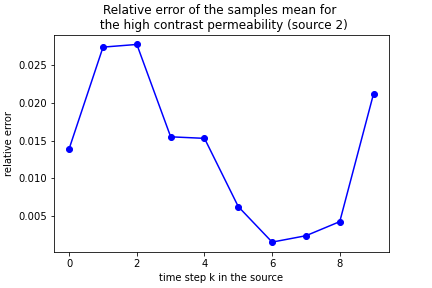}
\caption{Samples mean and the relative error of the high contrast permeability with the second source. Left: samples mean and true source at each time step $k$.  Right: relative error of the samples mean.} 
\label{numerical_exp2_high}
\end{figure}

From the four examples presented above, we conclude that our algorithm is able to trace the exact source positions even if the proposal is far away from the exact route. 
The next challenge for us is to reduce the amount of observation data since by the theory, we can trace the source with flux on any open subset of the boundary. We will study this topic in the future.

\section{Concluding remarks and future works.}\label{section_concluding}
In this work, we consider the inverse source problem in the parabolic equation. The unknown source has a semi-discrete formulation, which can be used to approximate the general form $F(x,t)$. We prove the uniqueness theorem--Theorem \ref{uniqueness}, which says the data from any nonempty open  subset of the boundary can support the uniqueness of the source. This conclusion is of significance in practical applications since it indicates that the source can be recovered from sparse boundary data and then save the cost. For the theoretical analysis, there is an interesting and meaningful work in the future, which is to determine the minimal observed area. Theorem \ref{uniqueness} illustrates that the nonempty open subset of the boundary is sufficient to support the uniqueness. However, only from this theorem, we do not know whether we can minimize the observed area further. To find the minimal observed area is one of our future works. 
  
In the aspect of numerical reconstructions, we propose an oil diffusion problem which is modelled by the parabolic equation.
One needs to trace the positions of the oil tanker given the flux of the boundary.
Since the model problem is a sequential prediction with a sequence of observations up to the time $n$, it is natural to formulate the problem to the Bayesian filtering problem. We apply one particle filter algorithm to solve the inverse  problem with two multiscale permeabilities.
Evaluating the likelihood function is the most time consuming step since it requires solving a forward problem with a fine resolution solver.
We use the finite element solver in this work, however, the solver can be improved by the multiscale finite element solver or the deep learning solver. These solvers are more efficient meanwhile preserving accuracy and we will study these topics in the future. We can also apply deep learning to train a solver. This is equivalent to solving the stochastic parametric PDE with the deep neural network. More precisely, we can train a mapping from the source to the flux on the boundary. One benefit of the method is: it will increase the computation efficiency in evaluating the likelihood function. Another benefit is: we can linearize the network and hence more advanced particle filter algorithms can be used.

\bibliographystyle{plainurl} 
\bibliography{note}

\appendix
\section{Particle filter algorithm.}
\label{app_sir}
Here we give the particle filter algorithm (PF) which we used in the simulation.

\begin{algorithm}[H]
\SetAlgoLined
Initialize the number of time steps $n$ and the number of particles $N$\;
Step 0: at $t = 0$, draw the states $x_0^{i}$ for $i = 1, ..., N$ from the prior $p(x_0)$\;
\For{$t = 1,2, ... , n$}
{
\For{$i = 1, ..., N$}
{
Sample $\hat{x}^{i}_t$ from the prior $q(x_t|x_{0:t-1}^i, y_{1:t})$ and extend the current trajectory by adding the temporary proposed state $\hat{x}_{0:t}^i = (x_{0:t-1}^i, \hat{x}_t^i)$\;
Calculate the importance weights recursively and normalize the resulting weights as  
$$
w_t^i = w_{t-1}^i\frac{p(y_t|\hat{x}_t^i ) p(x_t^i|x_{t-1}^i)}{ q(\hat{x}_t^i| x_{0:t-1}^{i}, y_{1:t})}.\;
$$
}
\For{$i = 1, ..., N$}
{
Normalize the importance weights for the resampling purpose as
$$
\Tilde{w}_t^i = \frac{w_t^i}{\sum_{j = 1}^N w_t^j}.\;
$$
}
Resampling: multiply samples $\hat{x}_{0:t}^i$ with the normalized resampling weight $\Tilde{w}_t^i$ to obtain $N$ random samples $x_{0:t}^i$ which is roughly distributed following $p(x_{0:t}^i|y_{1:t})$\;
Set $w_t^i = \Tilde{w}_t^i = \frac{1}{N}$ for all $i = 1, ..., N$\;
The algorithm will finally return a set of samples whose average is an approximation of the target distribution:
$$p(x_{0:t}|y_{0:t})\approx \hat{p}(x_{0:t}|y_{0:t}) = \frac{1}{N}\sum_{i = 1}^N \delta_{(x_{0:t}^i)}(dx_{0:t}).\;
$$
\caption{Sequential importance resampling (SIR)}
}
\end{algorithm}

It should be noted that we use the systematic resampling method in this work. To implement this method, one first samples $U_1\sim \mathcal{U}[0, \frac{1}{N}]$ and defines $U_i = U_1+\frac{i-1}{N}$ for $i = 2, ..., N$, then we set 
$$
N_n^i = \big|\{ U_j: \sum_{k = 1}^{i-1} W_n^k\leq U_j\leq \sum_{k = 1}^{i} W_n^k\}\big|
$$
with the convention $\sum_{k=1}^0:= 0$.

\end{document}